\documentclass[3p, number, sort&compress, times, lefttitle, dvipsnames]{elsarticle}

\journal{Computer Methods in Applied Mechanics and Engineering}

\usepackage{hyphenat}
\usepackage{subcaption}
\usepackage{placeins}
\usepackage{multirow}
\usepackage{algorithm}
\usepackage{algpseudocode}
\usepackage{graphicx}
\graphicspath{ {./figs/} }
\usepackage{tikz}
\usetikzlibrary{arrows.meta,calc,patterns}
\definecolor{pal29}{rgb}{0.890, 0.953, 0.984}
\definecolor{pal58}{rgb}{0.184, 0.310, 0.455}
\usepackage{amsmath,amsfonts}
\renewcommand{\Re}{\mathbb{R}}

\usepackage{amsthm}
\theoremstyle{plain}
\newtheorem*{proposition}{Proposition}

\theoremstyle{definition}
\usepackage{bm}
\newcommand{\vm}[1]{\bm{#1}}
\newcommand{\vx}{\vm{x}}

\usepackage{mathtools}

\usepackage[colorlinks=true]{hyperref}

\newcommand{\tref}[1]{Table~\ref{#1}}
\newcommand{\fref}[1]{Fig.~\ref{#1}}
\newcommand{\sref}[1]{Section~\ref{#1}}

\begin{document}

\title{Scaled boundary cubature scheme in higher dimensions: Integration over polytopes and curved regions}

\author[1]{Eric B.\ Chin\corref{cor1}}
\ead{chin23@llnl.gov}

\author[2]{N.\ Sukumar}

\cortext[cor1]{Corresponding author}

\address[1]{Lawrence Livermore National Laboratory, 7000 East Avenue, Livermore, CA 94550, USA}

\address[2]{Department of Civil and Environmental Engineering, University of California, Davis, CA 95616, USA}

\begin{abstract}
We extend the scaled boundary cubature (SBC) scheme from planar regions to higher-dimensional regions described by
oriented boundary patches.  The resulting parametrization transforms integrals over compact regions in $\Re^d$ into sums
of integrals over the boundary-patch parameter domains and a radial coordinate.  In three dimensions, this yields a
direct volume-integration rule for solids bounded by affine faces, triangular or tensor-product surface patches,
B-spline patches, NURBS patches, and combinations of curved and affine boundary representations.  For affine polytopes,
recursive application of the scaled boundary map yields nested tensor-product rules over simplex sectors; in three
dimensions, these reduce to tetrahedral-sector rules that apply equally to convex and nonconvex oriented polyhedra.  We
also develop transformations for weakly singular integrands.  Placing the scaling center at a point singularity exposes
the radial power canceled by the Jacobian, while generalized radial scalings and Gauss--Jacobi quadrature handle
fractional powers.  A transverse scaled-boundary map provides the analogous construction for affine singular sets, with
straight-line examples in three dimensions.  Numerical examples verify polynomial exactness on affine polyhedra and a
four-dimensional tesseract, rapid convergence on curved B-spline and NURBS solids, and the expected convergence
improvements for point and line singularities.  Near-boundary singularity tests also identify when additional patch
subdivision or patch-parameter transformations are required.
\end{abstract}


\maketitle

\section{Introduction}\label{sec:intro}

Computational mechanics discretizations increasingly require the integration of functions over domains that are not
simplices or tensor\hyp{}product cells.  Such domains arise in finite element and discontinuous Galerkin methods on
polygonal and polyhedral meshes~\cite{Cangiani:2014:HVD}, in the virtual element method
(VEM)~\cite{BeiraodaVeiga:2013:BPV,BeiraodaVeiga:2019:TVE}, in extended, embedded, and unfitted finite element methods
for interfaces, cracks, and immersed
boundaries~\cite{Moes:1999:AFE,Sukumar:2015:EFE,Burman:2014:CDG,Schillinger:2015:TFC}, in isogeometric analysis with
trimmed or boundary\hyp{}representation geometries~\cite{Hughes:2005:IAC,Marussig:2018:ICI,Chen:2016:ANB}, and in mortar
and contact formulations~\cite{Puso:2004:AMS,Hesch:2011:TTD,Farah:2015:SBE}.  Related integration tasks arise in
isogeometric boundary element methods, where spline-tailored rules based on quasi-interpolation, modified moments, and
singularity extraction have been developed for singular and nearly singular integrals over B-spline and NURBS boundary
patches~\cite{Calabro:2018:EQR,Aimi:2020:QFS,Falini:2021:CRB,DegliEsposti:2023:IBH}.  They also occur outside Galerkin
discretizations, for example in computer graphics simulations with nonconvex bodies~\cite{Guendelman:2003:NRB} and in
computer\hyp{}aided design when computing surface or volumetric moments from parametric
geometry~\cite{Krishnamurthy:2011:AGA}.  The integration problem in these settings is not limited to high\hyp{}order
polynomial cubature.  The boundary may be nonconvex, may be represented by polynomial or rational parametric patches, or
may be produced by intersections and trimming operations; the integrand may also contain nonpolynomial geometry factors
or weak singularities.  In this setting, the value of SBC is not only that it generates cubature points and weights, but
also that it provides a constructive parametrization of the enclosed region from its oriented boundary representation.

In Chin and Sukumar~\cite{Chin:2021:SBC}, we introduced the scaled boundary cubature (SBC) scheme for planar regions
bounded by affine and curved edges.  The central observation was that the scaled boundary parametrization produces a
tensor\hyp{}product integral over a radial coordinate and a boundary parameter.  This viewpoint also showed that the
homogeneous numerical integration (HNI) method~\cite{Lasserre:1998:ICP,Lasserre:1999:IHF,Chin:2015:NIH} is recovered as
a special case when the integrand is homogeneous.  The present paper extends that idea beyond planar regions: we develop
the higher\hyp{}dimensional construction, specialize it to three\hyp{}dimensional regions bounded by parametric
surfaces, and introduce transformations for weakly singular integrands.

Several established strategies are available for integration over nonstandard domains.  A direct approach is to
partition the domain into triangles, tetrahedra, or curved subcells and then apply standard quadrature rules on each
piece~\cite{Moes:1999:AFE,Sevilla:2008:NEF,Sevilla:2011:TNE,Fries:2017:HOM,Artioli:2020:ACP}.  This strategy is robust
and familiar, but for complex, nonconvex, or curved domains the geometry generation can dominate the cost, and
high\hyp{}order rules over many subcells can produce far more points than are needed by the integrand.
Moment\hyp{}fitting and optimized rules provide more compact formulas; for example, generalized Gaussian rules have been
constructed for arbitrary polygons with interior nodes and positive weights~\cite{Mousavi:2010:GGQ}, and efficient rules
in two and higher dimensions can be generated by numerical optimization~\cite{Xiao:2009:ANA}.  Moment\hyp{}fitting
variants have also been developed for cut cells and large deformation settings~\cite{Duster:2020:SEM,Garhuom:2022:NMF}.
These rules are attractive when they are available, but their construction requires solving nonlinear systems and
commonly starts from an existing rule or partition.  One moment\hyp{}fitting strategy is to compress a dense candidate
rule or discrete measure into a smaller weighted rule that preserves the moments of a chosen polynomial
space~\cite{Sommariva:2015:CMD,Bauman:2020:CAC}.  Such compressed rules are attractive for nonlinear computations in
which the same integration rule is reused many times over an element, because the setup cost can be amortized while
subsequent residual, tangent, or constitutive evaluations use fewer points.

Another class of methods uses Green's theorem, Stokes's theorem, or the divergence theorem to reduce the dimension of
the integral.  Product Gauss and Gauss--Green cubature rules use Green's formula for polygons and curvilinear
geometries, and they have also been used for moment computation over arbitrary
geometries~\cite{Sommariva:2007:PGC,Sommariva:2009:GGC,Gunderman:2021:SMF,Gunderman:2021:HAM}.  These methods avoid
volume meshing and scale with the number of boundary pieces, although the resulting cubature points need not lie inside
the physical domain and the implementation depends on suitable antiderivative or boundary\hyp{}integral constructions.
SBC instead keeps the domain integral as a mapped tensor\hyp{}product integral, which can be direct to apply when the
boundary is represented by CAD or simulation geometry such as NURBS, B\'{e}zier patches, segmented overlap regions, or
nonconvex polytopes.  The HNI method~\cite{Lasserre:1998:ICP,Lasserre:1999:IHF,Chin:2015:NIH,Chin:2020:AEM} is also
based on dimension reduction and is particularly efficient for homogeneous functions and polynomials after decomposition
into homogeneous terms.  Its restriction to homogeneous structure, however, makes general nonpolynomial functions and
black\hyp{}box integrands less direct.  For implicitly defined, trimmed, and curved polyhedral geometries,
high\hyp{}order quadrature schemes based on local height functions, recursive decompositions, folded decompositions,
patchwise surface quadrature, or transport corrections have been
developed~\cite{Saye:2015:HOQ,Scholz:2019:NIT,Scholz:2021:UHO,Antolin:2022:RNI,Loibl:2023:PQT}; these methods address
important cases in unfitted and isogeometric analysis but are tied to their particular geometry descriptions.  Weakly
singular integration is a related challenge, and classical Duffy\hyp{}type and Gauss--Jacobi transformations provide
important tools for vertex and simplex
singularities~\cite{Duffy:1982:QOP,Mousavi:2010:GDT,Chernov:2012:ECG,Lv:2019:ASD}.

SBC has recently been used to compute weak\hyp{}form integrals in several virtual element formulations over polygonal
and polyhedral meshes, including curved elements and enriched formulations with weakly singular
functions~\cite{Chin:2024:VEM}.  In contrast to that application\hyp{}focused study, the purpose of the present paper is
to develop SBC itself as a general higher\hyp{}dimensional integration method.  Our contributions are as follows:
\begin{itemize}
  \item a scaled boundary cubature formula in $\Re^d$ for compact regions bounded by oriented hypersurface patches;
  \item an explicit three\hyp{}dimensional specialization for regions bounded by parametric surface patches, including
        tensor\hyp{}product and triangular patch parameterizations;
  \item recursive constructions for polytopes and for patch parameter domains that are not already canonical integration
        domains;
  \item a simplified affine\hyp{}polyhedron formula based on nested scaled boundary maps, yielding tetrahedral sector
        integrals in tensor\hyp{}product coordinates;
  \item transformations for point singularities in arbitrary dimension, including the cancellation provided by the
        radial Jacobian factor and generalized scaled boundary or Gauss--Jacobi variants; and
  \item a transverse SBC\hyp{}like transformation for affine singular sets, with straight-line examples in three
  dimensions.
\end{itemize}

The remainder of this paper is organized as follows.  In~\sref{sec:sbc}, we present the higher\hyp{}dimensional scaled
boundary parametrization, its cubature rule, and specializations to polytopes and three\hyp{}dimensional curved regions.
In~\sref{sec:singular}, we develop transformations for weakly singular integrands with point and line singularities.
Numerical studies in~\sref{sec:results} examine polynomial exactness on affine polyhedra and a four-dimensional affine
polytope, convergence for smooth nonpolynomial functions, curved surface geometries, and singular integration.
Concluding remarks are given in~\sref{sec:conclusion}.

\section{Higher dimensional scaled boundary cubature}\label{sec:sbc}

\subsection{General scaled boundary parametrization in $\Re^d$}
\label{sec:sbc-general}

Let $\mathcal{R} \subset \Re^d$ be a compact region whose boundary is represented by oriented hypersurface patches
$\mathcal{S}_i$, $i = 1, \dotsc, m$.  Each patch is given by a sufficiently smooth parametrization $\vm{s}_i : D_i
\rightarrow \mathcal{S}_i$, where $D_i \subset \Re^{d-1}$.  The patch orientations are taken to be consistent with the
outward orientation of $\partial \mathcal{R}$; in particular, the determinant appearing below is a signed quantity.  We
assume that each parameter domain $D_i$ is either a canonical integration domain, such as a hypercube or simplex, or has
a boundary representation that permits the scaled boundary construction to be applied recursively.  The patch
parametrizations and the maps introduced below need not be globally injective, provided the oriented image counts the
region with the correct multiplicity.

For a scaling center $\vx_0 \in \Re^d$, define the scaled boundary (SB) map:
\begin{equation}\label{eq:sbc-general-map}
  \vm{\varphi}_i(\rho,\vm{t})
    = \vx_0 + \rho\bigl(\vm{s}_i(\vm{t})-\vx_0\bigr),
  \qquad 0 \leq \rho \leq 1,\quad \vm{t} \in D_i .
\end{equation}
Apart from the translation by $\vx_0$, this is the scaled boundary parametrization introduced by
Song~\cite{Song:1997:TSB}.  The coordinate $\rho$ linearly scales each boundary point to the scaling center.  If
$D_i=[0,1]^{d-1}$, then \eqref{eq:sbc-general-map} maps the unit hypercube to a curved pyramid with apex $\vx_0$ and
base $\mathcal{S}_i$; for more general $D_i$, it maps $[0,1]\times D_i$ to the corresponding swept sector
$\mathcal{R}^+_i$ (see \fref{fig:sbc-swept-sector}).  We denote the union of these swept sectors by $\mathcal{R}^+$.  If
$\mathcal{R}$ is star-convex with respect to $\vx_0$, then
$\mathcal{R}^+=\mathcal{R}$ up to shared sector boundaries.  Otherwise $\mathcal{R}^+$ may strictly contain
$\mathcal{R}$, and the signed contributions in \eqref{eq:sbc-general-integral} provide the cancellation needed to
recover the oriented integral over $\mathcal{R}$.

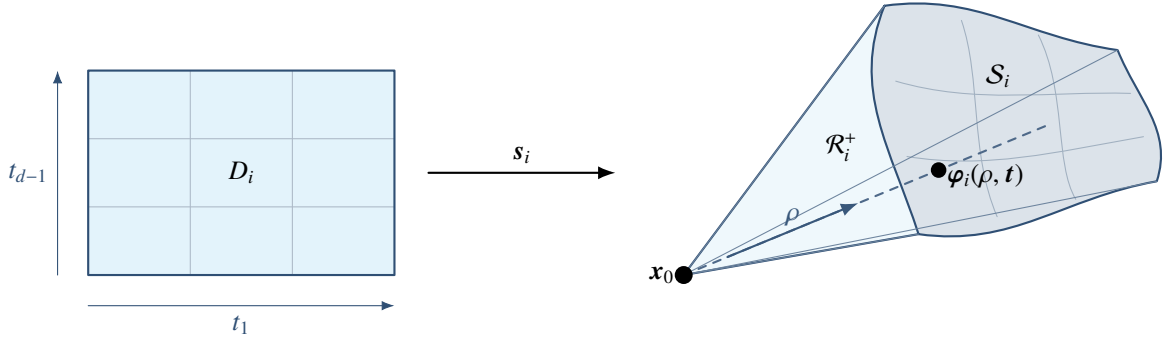
\begin{figure}[t]
  \centering
  \begin{tikzpicture}[scale=2.25, line cap=round, line join=round, >=Latex]
    \begin{scope}[xshift=-2.55cm]
      \fill[pal29] (0,0) rectangle (1.8,1.2);
      \draw[pal58!40] (0.6,0) -- (0.6,1.2);
      \draw[pal58!40] (1.2,0) -- (1.2,1.2);
      \draw[pal58!40] (0,0.4) -- (1.8,0.4);
      \draw[pal58!40] (0,0.8) -- (1.8,0.8);
      \draw[pal58, thick] (0,0) rectangle (1.8,1.2);
      \draw[pal58, ->] (0,-0.18) -- (1.8,-0.18) node[midway, below] {$t_1$};
      \draw[pal58, ->] (-0.18,0) -- (-0.18,1.2) node[midway, left] {$t_{d-1}$};
      \node at (0.9,0.6) {$D_i$};
    \end{scope}
    \draw[->, thick] (-0.55,0.6) -- (0.55,0.6) node[midway, above] {$\vm{s}_i$};
    \begin{scope}[xshift=0.95cm]
      \coordinate (x0) at (0,0);
      \coordinate (a) at (1.18,1.58);
      \coordinate (b) at (2.55,1.32);
      \coordinate (c) at (2.78,0.55);
      \coordinate (d) at (1.38,0.24);
      \coordinate (spt) at (2.14,0.88);
      \fill[pal29, opacity=0.55] (x0) -- (a)
        to[out=8,in=172] (b)
        to[out=-68,in=72] (c)
        to[out=190,in=-8] (d) -- cycle;
      \fill[pal58!20, opacity=0.72] (a)
        to[out=8,in=172] (b)
        to[out=-68,in=72] (c)
        to[out=190,in=-8] (d)
        to[out=108,in=-112] (a);
      \draw[pal58, thick] (a)
        to[out=8,in=172] (b)
        to[out=-68,in=72] (c)
        to[out=190,in=-8] (d)
        to[out=108,in=-112] (a);
      \draw[pal58!45] ($(a)!0.33!(d)$) .. controls (1.72,1.0) and (2.15,1.08) .. ($(b)!0.33!(c)$);
      \draw[pal58!45] ($(a)!0.66!(d)$) .. controls (1.78,0.62) and (2.18,0.72) .. ($(b)!0.66!(c)$);
      \draw[pal58!45] ($(a)!0.33!(b)$) .. controls (1.75,1.15) and (1.68,0.72) .. ($(d)!0.33!(c)$);
      \draw[pal58!45] ($(a)!0.66!(b)$) .. controls (2.25,1.08) and (2.18,0.65) .. ($(d)!0.66!(c)$);
      \draw[pal58, thick] (x0) -- (a);
      \draw[pal58, thick] (x0) -- (d);
      \foreach \p in {a,b,c,d} {
        \draw[pal58!70] (x0) -- (\p);
      }
      \draw[pal58!85, dashed, thick] (x0) -- (spt);
      \fill (x0) circle (1.5pt) node[left] {$\vx_0$};
      \fill ($(x0)!0.7!(spt)$) circle (1.3pt) node[right, yshift=-2pt] {$\vm{\varphi}_i(\rho,\vm{t})$};
      \node[above] at (1.85,1.05) {$\mathcal{S}_i$};
      \node at (0.93,0.75) {$\mathcal{R}^+_i$};
      \draw[->, pal58!95, thick] ($(x0)!0.12!(spt)$) -- ($(x0)!0.48!(spt)$) node[midway, above, yshift=-2pt] {$\rho$};
    \end{scope}
  \end{tikzpicture}
  \caption{Scaled boundary map from a patch parameter domain $D_i$ to a swept sector $\mathcal{R}^+_i$.  The grid on
  $D_i$ maps to corresponding parameter lines on the boundary patch $\mathcal{S}_i$.  The dashed line shows the radial
  segment obtained by fixing $\vm{t}$ and varying $\rho\in[0,1]$, and the marked point is $\vm{\varphi}_i(\rho,\vm{t})$
  for one intermediate value of $\rho$.}
  \label{fig:sbc-swept-sector}
\end{figure}

The Jacobian of \eqref{eq:sbc-general-map} is
\begin{equation}\label{eq:sbc-general-jacobian}
  J_i(\rho,\vm{t})
  = \det \nabla\vm{\varphi}_i(\rho,\vm{t})
  = \rho^{d-1}
    \det\Bigl[
      \vm{s}_i(\vm{t})-\vx_0 \quad
      \frac{\partial\vm{s}_i}{\partial t_1}(\vm{t}) \quad \cdots \quad
      \frac{\partial\vm{s}_i}{\partial t_{d-1}}(\vm{t})
    \Bigr].
\end{equation}
The factor $\rho^{d-1}$ is the key dimensional scaling: the $d-1$ directions tangent to the boundary patch are scaled by
$\rho$, whereas the direction associated with changing $\rho$, from $\vx_0$ to $\vm{s}_i(\vm{t})$, is not.

\begin{proposition}
Let the oriented patches $\mathcal{S}_i$ represent $\partial\mathcal{R}$, and let $f$ be smooth on an open set
containing $\mathcal{R}^+$.  Then
\begin{equation}\label{eq:sbc-general-integral}
  \int_{\mathcal{R}} f(\vx)\,d\vx
  =
  \sum_{i=1}^{m}\int_{0}^{1}\int_{D_i}
    f\bigl(\vm{\varphi}_i(\rho,\vm{t})\bigr)
    J_i(\rho,\vm{t})\,d\vm{t}\,d\rho .
\end{equation}
\end{proposition}
\begin{proof}
For $\vm{y}=\vx-\vx_0$, define
\[
  \vm{F}(\vx)=
  \vm{y}\int_0^1 f(\vx_0+\tau\vm{y})\tau^{d-1}\,d\tau .
\]
Using the product and chain rules,
\[
  \nabla\cdot\vm{F}(\vx)
  =
  \int_0^1
  \left[
    d\tau^{d-1}f(\vx_0+\tau\vm{y})
    +
    \tau^d\nabla f(\vx_0+\tau\vm{y})\cdot\vm{y}
  \right]d\tau
  =
  \int_0^1\frac{d}{d\tau}
  \left[\tau^d f(\vx_0+\tau\vm{y})\right]\,d\tau
  = f(\vx).
\]
The divergence theorem gives
\[
  \int_{\mathcal{R}} f(\vx)\,d\vx
  =
  \sum_{i=1}^m\int_{D_i}\vm{F}(\vm{s}_i(\vm{t}))\cdot\vm{N}_i(\vm{t})\,d\vm{t},
\]
where $\vm{N}_i\,d\vm{t}$ is the oriented normal measure induced by $\vm{s}_i$.  For any $\vm{a}\in\Re^d$,
\[
  \vm{a}\cdot\vm{N}_i(\vm{t})
  =
  \det\Bigl[
    \vm{a} \quad
    \frac{\partial\vm{s}_i}{\partial t_1}(\vm{t}) \quad \cdots \quad
    \frac{\partial\vm{s}_i}{\partial t_{d-1}}(\vm{t})
  \Bigr].
\]
Taking $\vm{a}=\vm{s}_i(\vm{t})-\vx_0$, substituting the definition of $\vm{F}$, and identifying
$\vx_0+\tau(\vm{s}_i(\vm{t})-\vx_0)=\vm{\varphi}_i(\tau,\vm{t})$ leads to \eqref{eq:sbc-general-integral}, with the
dummy variable $\tau$ renamed as $\rho$.
\end{proof}

Equation~\eqref{eq:sbc-general-integral} is also closely connected to the radial integration method (RIM), which
transforms domain integrals into boundary integrals containing a radial primitive~\cite{Gao:2002:RIM}.  If
$R=\|\vm{s}_i-\vx_0\|$, the normalized substitution $r=\rho R$ in that primitive yields precisely the radial and
boundary factors in \eqref{eq:sbc-general-integral}.  RIM absorbs the radial integral into the boundary integrand,
whereas SBC retains the radial coordinate explicitly to construct tensor\hyp{}product cubature points and weights.

When $\mathcal{R}$ is star-convex with respect to an interior $\vx_0$, the swept sectors form a partition of
$\mathcal{R}$ up to shared boundaries, and the signed Jacobians have one sign when the patch orientations are
consistent.  This is the setting in which positive cubature weights are expected.  For nonconvex regions, or for scaling
centers outside the region, $\mathcal{R}^+$ can extend outside $\mathcal{R}$ (see
\fref{fig:sbc-star-nonconvex}); the formula remains an oriented identity
and the resulting cubature weights may be signed.  In that case, $f$ must be evaluable on the swept region and the
integral is obtained by cancellation.  The map is not one-to-one at $\rho=0$, since every patch parameter maps to
$\vx_0$, but this set has zero measure.  Thus the degeneracy at the scaling center is not a practical obstruction for
cubature, although it can matter in boundary-value discretizations based directly on the same parametrization
\cite{Arioli:2019:SBP}.

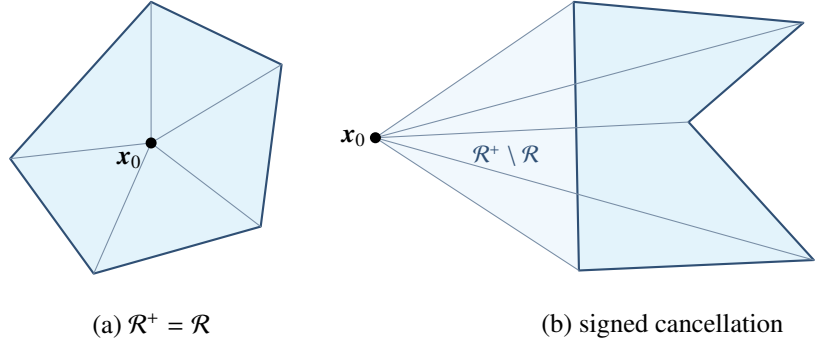
\begin{figure}[t]
  \centering
  \begin{tikzpicture}[scale=1.38, line cap=round, line join=round, >=Latex]
    \begin{scope}[xshift=-2.75cm]
      \coordinate (x0) at (0,0);
      \coordinate (p1) at (0.0,1.35);
      \coordinate (p2) at (1.25,0.75);
      \coordinate (p3) at (1.05,-0.8);
      \coordinate (p4) at (-0.55,-1.25);
      \coordinate (p5) at (-1.35,-0.15);
      \fill[pal29] (p1) -- (p2) -- (p3) -- (p4) -- (p5) -- cycle;
      \draw[pal58, thick] (p1) -- (p2) -- (p3) -- (p4) -- (p5) -- cycle;
      \foreach \p in {p1,p2,p3,p4,p5} {
        \draw[pal58!65] (x0) -- (\p);
      }
      \fill (x0) circle (1.5pt) node[below left] {$\vx_0$};
      \node at (0,-1.75) {(a) $\mathcal{R}^+=\mathcal{R}$};
    \end{scope}
    \begin{scope}[xshift=2.15cm]
      \coordinate (x0) at (-2.75,0.05);
      \coordinate (q1) at (-0.85,1.35);
      \coordinate (q2) at (1.35,1.15);
      \coordinate (q3) at (0.25,0.2);
      \coordinate (q4) at (1.45,-1.12);
      \coordinate (q5) at (-0.8,-1.22);
      \fill[pal29!55, opacity=0.65] (x0) -- (q1) -- (q2) -- cycle;
      \fill[pal29!55, opacity=0.65] (x0) -- (q2) -- (q3) -- cycle;
      \fill[pal29!55, opacity=0.65] (x0) -- (q3) -- (q4) -- cycle;
      \fill[pal29!55, opacity=0.65] (x0) -- (q4) -- (q5) -- cycle;
      \fill[pal29!55, opacity=0.65] (x0) -- (q5) -- (q1) -- cycle;
      \fill[pal29] (q1) -- (q2) -- (q3) -- (q4) -- (q5) -- cycle;
      \draw[pal58, thick] (q1) -- (q2) -- (q3) -- (q4) -- (q5) -- cycle;
      \draw[pal58!65] (x0) -- (q1);
      \draw[pal58!65] (x0) -- (q2);
      \draw[pal58!65] (x0) -- (q3);
      \draw[pal58!65] (x0) -- (q4);
      \draw[pal58!65] (x0) -- (q5);
      \fill (x0) circle (1.5pt) node[left] {$\vx_0$};
      \node[pal58, font=\small] at (-1.5,-0.125) {$\mathcal{R}^+\setminus\mathcal{R}$};
      \node at (0,-1.75) {(b) signed cancellation};
    \end{scope}
  \end{tikzpicture}
  \caption{In the star-convex case, the swept sectors partition $\mathcal{R}$.  For a nonconvex region or exterior
  scaling center, $\mathcal{R}^+$ can include points outside $\mathcal{R}$, and the oriented sum cancels the extra
  contributions.}
  \label{fig:sbc-star-nonconvex}
\end{figure}

\subsection{Cubature rules and polynomial exactness}
\label{sec:sbc-cubature}

Equation \eqref{eq:sbc-general-integral} reduces the volume integral to a sum of integrals over simple product domains
whenever the patch domains $D_i$ are simple.  If $D_i=[0,1]^{d-1}$, a tensor-product Gauss rule in $\rho$ and the $d-1$
patch parameters gives
\begin{equation}\label{eq:sbc-product-rule}
  \int_{\mathcal{R}} f(\vx)\,d\vx
  \approx
  \sum_{i=1}^{m}\sum_{q_\rho}\sum_{\vm{q}}
  w_{q_\rho}w_{\vm{q}}\,
  f\bigl(\vm{\varphi}_i(\rho_{q_\rho},\vm{t}_{\vm{q}})\bigr)
  J_i(\rho_{q_\rho},\vm{t}_{\vm{q}}),
\end{equation}
where $w_{\vm{q}}$ denotes the product of the one-dimensional weights in the patch coordinates.  The point count scales
as $O(mp^d)$ for $m$ patches and a $p$-point rule in each coordinate.  If $D_i$ is a simplex, one may use a standard
simplex rule on $D_i$ or apply SBC recursively to the parameter domain.  The recursive option is useful when $D_i$ is a
polygonal, polyhedral, or trimmed domain whose boundary is already available in parametric form.  In the affine
polyhedron specialization below, a tensor-product rule with orders $p_\rho$, $p_t$, and $p_u$ uses $p_\rho p_t p_u\sum_i
n_i$ points, where $n_i$ is the number of edges of face $\mathcal{S}_i$.

The exactness of \eqref{eq:sbc-product-rule} follows from the mapped integrand.  If $f$ is polynomial and each
$\vm{s}_i$ is polynomial, then $f(\vm{\varphi}_i(\rho,\vm{t}))J_i(\rho,\vm{t})$ is a polynomial in the integration
parameters, and a sufficiently high-order rule integrates it exactly.  This includes affine polytopes and polynomial
surface patches.  For rational patches, including NURBS, the mapped integrand is generally rational even when $f$ is
polynomial, so polynomial exactness is not expected; for smooth rational geometry and smooth $f$, the tensor-product
rules still give rapid convergence as shown in \sref{sec:ex-curved-solids}.

\subsection{Three-dimensional regions bounded by parametric surfaces}
\label{sec:sbc-3d-surfaces}

For $\mathcal{R}\subset\Re^3$, the boundary patches are surfaces
$\vm{s}_i(t,u):D_i\rightarrow\mathcal{S}_i$.  The SB map becomes
\begin{equation}\label{eq:sbc-surface-map}
  \vm{\varphi}_i(\rho,t,u)
  = \vx_0+\rho\bigl(\vm{s}_i(t,u)-\vx_0\bigr),
  \qquad 0\leq \rho\leq 1,\quad (t,u)\in D_i ,
\end{equation}
and its Jacobian is
\begin{equation}\label{eq:sbc-surface-jacobian}
  J_i(\rho,t,u)
  =
  \rho^2
  \bigl(\vm{s}_i(t,u)-\vx_0\bigr)\cdot
  \left(
    \frac{\partial\vm{s}_i}{\partial t}(t,u)
    \times
    \frac{\partial\vm{s}_i}{\partial u}(t,u)
  \right).
\end{equation}
The order of the cross product is part of the surface orientation.  Reversing the patch orientation reverses the sign of
the contribution.  For a star-convex solid and an interior scaling center, outward-oriented surface patches give
positive signed volumes.

Tensor-product Bezier, B-spline, and NURBS patches have $D_i=[0,1]^2$ and fit directly into \eqref{eq:sbc-product-rule}.
Triangular Bezier patches and other simplex-parametrized surfaces can be integrated with a triangular rule on $D_i$, or
with a secondary SB construction on the triangular parameter domain.  Trimmed patches can also be used, provided that
the trimmed parameter region is integrated by a compatible rule.  For polynomial surface patches and polynomial $f$,
\eqref{eq:sbc-surface-jacobian} again yields a polynomial mapped integrand.  For NURBS and other rational patches, the
mapped integrand is rational, so the numerical studies below assess convergence rather than polynomial exactness.

\subsection{Recursive integration over polytopes}
\label{sec:sbc-recursive-polytopes}

Let $\mathcal{P}\subset\Re^d$ be a polytope whose facets $\mathcal{P}_i$ are $(d-1)$-dimensional polytopes lying in
affine hyperplanes $\mathcal{H}_i$.  Applying \eqref{eq:sbc-general-map} to each facet reduces the integral over
$\mathcal{P}$ to integrals over the facets.  If a facet is not already a canonical integration domain, a second SB map
is applied within $\mathcal{H}_i$ from a facet scaling center $\vx_i\in\mathcal{H}_i$ to the boundary of
$\mathcal{P}_i$.  The process continues until the remaining domains are edges, simplices, hypercubes, or other domains
with available rules.

The nested construction is most simply viewed as a sequence of affine or curved radial maps.  Each level contributes a
radial coordinate and the corresponding radial power: the first level in $\Re^d$ contributes $\rho^{d-1}$, the next
level on a facet contributes a factor of degree $d-2$, and so on.  The algorithm is direct: integrate over the oriented
boundary of the current polytope in its own affine hull, using a scaling center in that affine hull, and multiply the
Jacobians from the nested maps.  The same oriented interpretation applies at every level.  For star-convex choices of
scaling centers the sectors partition the polytope; otherwise, signed sectors represent the integral by cancellation.
Algorithm~\ref{alg:recursive-sbc-polytope} summarizes this recursive construction.

\begin{algorithm}[t]
\caption{Recursive scaled boundary cubature over a polytope}
\label{alg:recursive-sbc-polytope}
\begin{algorithmic}[1]
\Require An oriented $k$\hyp{}dimensional polytope $\mathcal{P}^{(k)}$, scaling center
  $\vx_c^{(k)}$ in its affine hull, and an accumulated map $\vm{\Phi}$ and Jacobian factor $J$.
\Ensure Physical cubature points and signed weights for the branch rooted at $\mathcal{P}^{(k)}$.
\If{$\mathcal{P}^{(k)}$ is a canonical integration domain}
  \State Apply the corresponding rule on $\mathcal{P}^{(k)}$.
  \State Compose its points with $\vm{\Phi}$ and multiply its weights by $J$.
  \State \Return the resulting physical points and signed weights.
\EndIf
\ForAll{oriented facets $\mathcal{P}^{(k-1)}_\ell$ of $\mathcal{P}^{(k)}$}
  \State Choose a facet scaling center $\vx_{c,\ell}^{(k-1)}$ in the affine hull of
    $\mathcal{P}^{(k-1)}_\ell$.
  \State Define the SB map from $\vx_c^{(k)}$ to $\mathcal{P}^{(k-1)}_\ell$ and its signed Jacobian factor.
  \State Compose this map and Jacobian factor with $\vm{\Phi}$ and $J$.
  \State Call Algorithm~\ref{alg:recursive-sbc-polytope} with arguments
    $\mathcal{P}^{(k-1)}_\ell$, $\vx_{c,\ell}^{(k-1)}$, the composed map, and the composed Jacobian factor.
\EndFor
\State \Return the union of the points and signed weights generated by all recursive branches.
\end{algorithmic}
\end{algorithm}

\subsection{Simplification for affine polyhedra}
\label{sec:sbc-affine-polyhedra}

For an affine polyhedron $\mathcal{P}\subset\Re^3$, the recursive construction reduces to tensor-product integrals over
tetrahedral sectors.  In this specialization, $\rho$ is the volume radial coordinate, $\eta$ is the face-radial
coordinate, and $t$ is the edge coordinate.  Let $\mathcal{S}_i$, $i=1,\dotsc,m$, be the oriented planar faces of
$\mathcal{P}$, and choose a face scaling point $\vx_i$ in the plane of $\mathcal{S}_i$.  Let the oriented boundary of
face $\mathcal{S}_i$ consist of edges with endpoint pairs $\vm{v}_{ij1}$ and $\vm{v}_{ij2}$, $j=1,\dotsc,n_i$, ordered
consistently with the face orientation.  An example is illustrated in \fref{fig:sbc-polyhedron-sector}.  The face sector
associated with edge $j$ is parametrized by

\begin{figure}[t]
  \centering
  \begin{tikzpicture}[scale=1.575, line cap=round, line join=round, >=Latex]
    \coordinate (x0) at (-0.55,-1.25);
    \coordinate (v1) at (-0.85,0.70);
    \coordinate (v2) at (1.45,0.92);
    \coordinate (v3) at (1.95,2.05);
    \coordinate (v4) at (-1.15,1.95);
    \coordinate (xi) at (0.05,1.35);
    \coordinate (edgept) at ($(v1)!0.58!(v2)$);
    \fill[pal29] (v1) -- (v2) -- (v3) -- (v4) -- cycle;
    \draw[pal58, thick] (v1) -- (v2) -- (v3) -- (v4) -- cycle;
    \draw[pal58!60] (xi) -- (v1);
    \draw[pal58!60] (xi) -- (v2);
    \fill[pal58!18] (x0) -- (xi) -- (v1) -- cycle;
    \fill[pal58!26] (x0) -- (xi) -- (v2) -- cycle;
    \fill[pal58!35] (x0) -- (v1) -- (v2) -- cycle;
    \draw[pal58, thick] (x0) -- (xi) -- (v1) -- cycle;
    \draw[pal58, thick] (x0) -- (v2) -- (xi);
    \draw[pal58, thick] (v1) -- (v2);
    \draw[pal58, dashed] (x0) -- (v3);
    \draw[pal58, dashed] (x0) -- (v4);
    \fill (x0) circle (1.5pt) node[below] {$\vx_0$};
    \fill (xi) circle (1.5pt) node[left, above] {$\vx_i$};
    \fill (v1) circle (1.5pt) node[left] {$\vm{v}_{ij1}$};
    \fill (v2) circle (1.5pt) node[right] {$\vm{v}_{ij2}$};
    \node[above] at (0.25,2.08) {$\mathcal{S}_i$};
    \node at (0.15,0.05) {$6V_{ij}$};
    \draw[->, pal58!95, thick] ($(xi)!0.18!(edgept)$) -- ($(xi)!0.62!(edgept)$) node[midway, right] {$\eta$};
    \draw[->, pal58!95, thick] ($(v1)!0.22!(v2)$) -- ($(v1)!0.58!(v2)$) node[midway, below] {$t$};
    \draw[->, pal58!95, thick] ($(x0)!0.14!(xi)$) -- ($(x0)!0.48!(xi)$) node[midway, left] {$\rho$};
  \end{tikzpicture}
  \caption{One tetrahedral sector in the affine-polyhedron construction.  The face map sweeps from the face point
  $\vx_i$ to the edge $(\vm{v}_{ij1},\vm{v}_{ij2})$, and the volume map sweeps that face sector from $\vx_0$.  The
  arrows indicate representative coordinate directions: $t$ moves along the edge, $\eta$ sweeps the face sector from
  $\vx_i$ toward points on the edge, and $\rho$ sweeps the resulting face sector through the tetrahedral volume.}
  \label{fig:sbc-polyhedron-sector}
\end{figure}
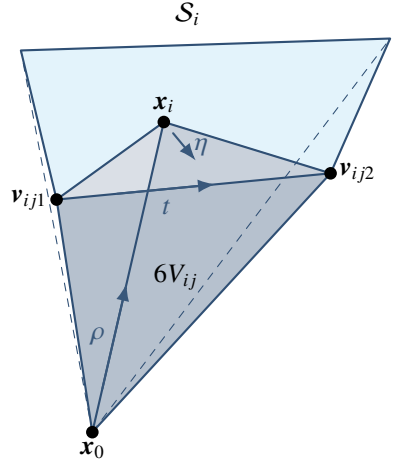

\begin{equation}\label{eq:sbc-face-sector-map}
  \vm{s}_{ij}(\eta,t)
  =
  \vx_i
  +
  \eta\left[
    (1-t)\vm{v}_{ij1}+t\vm{v}_{ij2}-\vx_i
  \right],
  \qquad 0\leq \eta,t\leq 1 .
\end{equation}
Nesting this map inside the volume SB map produces
\begin{equation}\label{eq:sbc-polyhedron-sector-map}
  \vm{\varphi}_{ij}(\rho,\eta,t)
  =
  \vx_0+\rho\bigl(\vm{s}_{ij}(\eta,t)-\vx_0\bigr),
  \qquad 0\leq \rho,\eta,t\leq 1 .
\end{equation}
Its Jacobian is
\begin{equation}\label{eq:sbc-polyhedron-sector-jacobian}
  J_{ij}(\rho,\eta,t)
  =
  \rho^2 \eta\,
  (\vx_i-\vx_0)\cdot
  \left[
    (\vm{v}_{ij1}-\vx_i)\times(\vm{v}_{ij2}-\vm{v}_{ij1})
  \right]
  =
  6 V_{ij}\,\rho^2 \eta ,
\end{equation}
where $V_{ij}$ is the signed volume of the tetrahedron with vertices
$\vx_0$, $\vx_i$, $\vm{v}_{ij1}$, and $\vm{v}_{ij2}$.  Therefore,
\begin{equation}\label{eq:sbc-polyhedron-integral}
  \int_{\mathcal{P}} f(\vx)\,d\vx
  =
  \sum_{i=1}^{m}\sum_{j=1}^{n_i}
  6V_{ij}
  \int_0^1\int_0^1\int_0^1
    f\bigl(\vm{\varphi}_{ij}(\rho,\eta,t)\bigr)\rho^2\eta\,dt\,d\eta\,d\rho .
\end{equation}
Equation \eqref{eq:sbc-polyhedron-integral} is a tensor-product cubature formula over each sector.  Choosing $\vx_0$ at
a vertex of the polyhedron or $\vx_i$ at a vertex of a face can eliminate some zero-volume or redundant sectors, but the
best choice depends on the face and edge counts and on the desired weight distribution.  For convex and star-convex
polyhedra, interior scaling centers lead to positive sector volumes when the boundary orientation is consistent.  For
nonconvex polyhedra or exterior choices of $\vx_0$, the signed volumes in \eqref{eq:sbc-polyhedron-integral} handle the
required cancellation, provided the integrand can be evaluated on the swept tetrahedral sectors.

\subsection{Relationship to HNI and previous SBC work}
\label{sec:sbc-hni}

The SB formula also clarifies the relationship with homogeneous numerical integration.  Suppose, for simplicity, that
$\vx_0=\vm{0}$ and that $f$ is homogeneous of degree $q$, so $f(\rho\vm{s})=\rho^q f(\vm{s})$.  The radial part of
\eqref{eq:sbc-general-integral} then contains $\rho^{q+d-1}$ and can be integrated analytically:
\[
  \int_0^1 \rho^{q+d-1}\,d\rho = \frac{1}{q+d}.
\]
This recovers the dimension-reduction structure used in HNI for homogeneous functions over polytopes
\cite{Lasserre:1998:ICP,Lasserre:1999:IHF,Chin:2015:NIH}.  In particular, the recursive face-to-edge construction in
Algorithm~\ref{alg:recursive-sbc-polytope} parallels the recursive polyhedron integration procedure in Section~3.1 of
Chin et al.~\cite{Chin:2015:NIH}: each step replaces an integral over a polytope by oriented integrals over its facets.
The distinction is that HNI evaluates the radial contribution analytically by exploiting homogeneity, whereas SBC
retains the radial variables and applies quadrature to them.  This removes the restriction to homogeneous integrands and
permits black-box smooth functions, nonpolynomial functions, and parametric curved geometries.  The two-dimensional
version of this connection was developed in detail in~\cite{Chin:2021:SBC}; the present construction is its
higher-dimensional extension.

\section{Integration of weakly singular functions}\label{sec:singular}

The scaled boundary map is particularly useful for weakly singular integrals when the singular set is compatible with
the scaling direction.  The Jacobian factor generated by the map can cancel part, or all, of the singular radial power,
thereby reducing the burden on the one-dimensional quadrature rule in the scaling coordinate.  This section summarizes
that mechanism for point singularities in arbitrary dimension and then introduces transverse scaling for affine singular
sets, with an explicit three-dimensional straight-line specialization.

\subsection{\texorpdfstring{Point singularities in $\Re^d$}{Point singularities in R-d}}\label{sec:point-rd}

Let $\mathcal{R}\subset\Re^d$ be represented by the oriented boundary patches used in \sref{sec:sbc-general}, and
consider
\begin{equation}\label{eq:singular-point-integral}
  I = \int_{\mathcal{R}} \frac{g(\vx)}{\|\vx-\vx_c\|^\beta}\,d\vx,
\end{equation}
where $g$ is smooth on an open set containing the swept region $\mathcal{R}^+$, as in the nonsingular case; the only
singular factor is the radial distance term.  A point singularity of this form is locally integrable if
\begin{equation}\label{eq:point-integrability}
  0 < \beta < d .
\end{equation}
We first assume that $\vx_c$ is an admissible scaling center and set $\vx_0=\vx_c$ in
\eqref{eq:sbc-general-map}.  Then
\begin{equation}\label{eq:point-distance-standard}
  \|\vm{\varphi}_i(\rho,\vm{t})-\vx_c\| =
  \rho\,\|\vm{s}_i(\vm{t})-\vx_c\|,
\end{equation}
and the Jacobian in \eqref{eq:sbc-general-jacobian} contributes $\rho^{d-1}$.  Hence, \eqref{eq:singular-point-integral}
becomes
\begin{equation}\label{eq:point-standard-sbc}
  I = \sum_i \int_0^1 \int_{D_i}
  g\bigl(\vm{\varphi}_i(\rho,\vm{t})\bigr)
  \frac{\rho^{d-1-\beta}\,
  \det\left[\vm{s}_i(\vm{t})-\vx_c,\partial_1\vm{s}_i,\dotsc,\partial_{d-1}\vm{s}_i\right]}
  {\|\vm{s}_i(\vm{t})-\vx_c\|^\beta}
  \,d\vm{t}\,d\rho .
\end{equation}
The singular radial dependence has been reduced to the power $\rho^{d-1-\beta}$.  In three dimensions, the common
kernels $1/r$ and $1/r^2$ therefore leave the powers $\rho$ and $\rho^0$, respectively.  The latter case removes the
radial singularity exactly, while the former yields a smooth radial factor for smooth $g$ and boundary maps.  The same
radial regularization mechanism is used in RIM treatments of singular domain integrals~\cite{Gao:2011:NEA}.

For fractional or stronger weak singularities, it can be advantageous to modify the radial scaling.  Define the
generalized scaled boundary map
\begin{equation}\label{eq:generalized-point-map}
  \vm{\varphi}_{i,\alpha}(\xi,\vm{t})
  = \vx_c + \xi^\alpha\bigl(\vm{s}_i(\vm{t})-\vx_c\bigr),
  \qquad 0\leq\xi\leq1,\quad \alpha>0 .
\end{equation}
Since $\rho=\xi^\alpha$, this map has Jacobian
\begin{equation}\label{eq:generalized-point-jacobian}
  J_{i,\alpha}(\xi,\vm{t}) =
  \alpha\,\xi^{\alpha d-1}
  \det\left[\vm{s}_i(\vm{t})-\vx_c,\partial_1\vm{s}_i,\dotsc,\partial_{d-1}\vm{s}_i\right] .
\end{equation}
Substitution into \eqref{eq:singular-point-integral} gives
\begin{equation}\label{eq:point-generalized-sbc}
  I = \sum_i \int_0^1 \int_{D_i}
  g\bigl(\vm{\varphi}_{i,\alpha}(\xi,\vm{t})\bigr)
  \frac{\alpha\,\xi^{\alpha(d-\beta)-1}\,
  \det\left[\vm{s}_i(\vm{t})-\vx_c,\partial_1\vm{s}_i,\dotsc,\partial_{d-1}\vm{s}_i\right]}
  {\|\vm{s}_i(\vm{t})-\vx_c\|^\beta}
  \,d\vm{t}\,d\xi .
\end{equation}
The exponent $\alpha(d-\beta)-1$ is nonnegative whenever $\alpha(d-\beta)\geq1$.  For polynomial integrands and affine
patches, choosing the smallest positive integer $\alpha$ such that $\alpha(d-\beta)$ is a positive integer removes
fractional radial powers while limiting the increase in polynomial degree caused by the composition
$g(\vm{\varphi}_{i,\alpha})$.  When this choice would make the mapped integrand unnecessarily high degree, an
alternative is to retain the standard map and use a Gauss--Jacobi rule in the radial coordinate with weight
$\rho^{d-1-\beta}$, or to use a Gauss--Jacobi rule in $\xi$ for the weight $\xi^{\alpha(d-\beta)-1}$; see, for example,
the convergence results for singular simplex integrals in \cite{Chernov:2012:ECG} and the two-dimensional SBC discussion
in \cite{Chin:2021:SBC}.

The formulas above remove only the radial singularity associated with the scaling center.  The remaining factor
$\|\vm{s}_i(\vm{t})-\vx_c\|^{-\beta}$ is a boundary-patch contribution.  If the singular point lies on, or very near,
one of the boundary patches, then this factor can be singular or nearly singular in the patch parameters.  In two
dimensions, this residual edge contribution can be treated by distance transformations for affine edges and related
near-singular integration techniques~\cite{Ma:2002:DTN,Lv:2019:ASD,Chin:2021:SBC}.  In higher dimensions, the same issue
becomes a lower-dimensional singular or nearly singular integral over the boundary-patch parameter domain.  Possible
remedies include patch subdivision, local patch-parameter transformations, or an alternative scaling center to isolate
the difficult feature.  For B-spline and NURBS patches, the spline quasi-interpolation, modified-moment, and
singularity-extraction rules developed for IgA-BEM provide another possible treatment of this lower-dimensional factor
\cite{Calabro:2018:EQR,Aimi:2020:QFS,Falini:2021:CRB,DegliEsposti:2023:IBH}.  The point-singularity studies in
\sref{sec:ex-point-singularities} first isolate the radial part of this behavior and then test one such local
patch-parameter transformation: a face-level transfer for the boundary-patch contribution that remains when the singular
point approaches a face.

\subsection{Point singularities in three-dimensional sectors}\label{sec:point-3d}

For a three-dimensional region bounded by parametric surfaces, the preceding expression takes a compact form.  Let
$\vm{s}_i(t,u)$ be an oriented surface patch and set $\vx_0=\vx_c$.  With the generalized map
\begin{equation}\label{eq:point-3d-map}
  \vm{\varphi}_{i,\alpha}(\xi,t,u)
  = \vx_c + \xi^\alpha\bigl(\vm{s}_i(t,u)-\vx_c\bigr),
\end{equation}
we obtain
\begin{equation}\label{eq:point-3d-integral}
  \int_{\mathcal{R}} \frac{g(\vx)}{\|\vx-\vx_c\|^\beta}\,d\vx
  = \sum_i \int_0^1 \int_{D_i}
  g\bigl(\vm{\varphi}_{i,\alpha}(\xi,t,u)\bigr)
  \frac{\alpha\,\xi^{\alpha(3-\beta)-1}
  \bigl(\vm{s}_i(t,u)-\vx_c\bigr)\cdot
  \left(\partial_t\vm{s}_i\times\partial_u\vm{s}_i\right)}
  {\|\vm{s}_i(t,u)-\vx_c\|^\beta}
  \,dt\,du\,d\xi .
\end{equation}
The standard choice $\alpha=1$ leaves the radial powers $\xi^{2-\beta}$.  Thus the Newtonian kernel $\beta=1$ leaves a
factor $\xi$, and the stronger weak singularity $\beta=2$ leaves no radial power.  If $\beta$ is fractional, or if a
higher-order smoothness condition is desired in the radial coordinate, the generalized map can be chosen so that
$\alpha(3-\beta)-1$ is a convenient nonnegative integer.  Otherwise, a weighted radial rule is often preferable because
it avoids over-resolving the composed function $g(\vx_c+\xi^\alpha(\vm{s}_i-\vx_c))$.

For affine polyhedra, \eqref{eq:sbc-polyhedron-integral} can be modified in the same manner by choosing the volume
scaling center $\vx_0=\vx_c$ and replacing the outer scaling coordinate by $\xi^\alpha$.  For a face sector
$\vm{s}_{ij}(t,u)$, this gives
\begin{equation}\label{eq:point-polyhedron-sector}
  \int_{\mathcal{P}} \frac{g(\vx)}{\|\vx-\vx_c\|^\beta}\,d\vx
  = \sum_i\sum_j 6V_{ij}\int_0^1\int_0^1\int_0^1
  g\bigl(\vx_c+\xi^\alpha(\vm{s}_{ij}(t,u)-\vx_c)\bigr)
  \frac{\alpha\,\xi^{\alpha(3-\beta)-1}t}
  {\|\vm{s}_{ij}(t,u)-\vx_c\|^\beta}
  \,du\,dt\,d\xi .
\end{equation}
The factor $t$ is inherited from the recursive face decomposition and is independent of the point-singularity
cancellation.  The face scaling points $\vx_i$ still influence the behavior of the surface factor
$\|\vm{s}_{ij}(t,u)-\vx_c\|^{-\beta}$ and should be chosen, together with any face subdivision, so that near-singular
features are not spread over large parameter intervals.

If $\vx_c$ lies outside $\mathcal{R}$, the integrand is smooth on the region but may be nearly singular.  Such integrals
are known to converge slowly under standard polynomial-precision rules when the source approaches the integration
domain~\cite{Chin:2017:MCD,Ma:2002:DTN}.  In this setting, placing the scaling center at $\vx_c$ may require evaluating
$g$ on swept sectors outside the physical region and can introduce signed weights.  This is permissible in the oriented
formulation when the integrand is defined on the swept region, but it is not always the most stable choice. Local
subdivision around the closest boundary feature, a local patch-parameter transformation, or an adaptive rule may be more
efficient for nearly singular exterior sources.  The numerical examples in \sref{sec:ex-point-singularities} include
both separated point singularities and a near-boundary point-source test that applies such a face-level parameter
transfer to the residual boundary factor.

\subsection{Transverse scaling for affine singular sets}\label{sec:line}

We consider a straight line singularity in three dimensions.  Choose a point on the line as the origin and an
orthonormal frame in which $\vm{e}_3$ is tangent to the line and $\vm{e}_1,\vm{e}_2$ span the transverse plane.  In
these adapted coordinates, write $\vx=x_1\vm{e}_1+x_2\vm{e}_2+x_3\vm{e}_3$.  The singular line is
$\{\lambda\vm{e}_3:\lambda\in\Re\}$, and the transverse distance is $r_\perp(\vx)=(x_1^2+x_2^2)^{1/2}$.  Since the local
transverse measure is two-dimensional, the kernel $r_\perp(\vx)^{-\beta}$ is locally integrable when
\begin{equation}\label{eq:transverse-integrability}
  0 < \beta < 2 .
\end{equation}
We therefore consider integrals of the form
\begin{equation}\label{eq:line-singular-integral}
  I = \int_{\mathcal{R}} \frac{g(\vx)}{r_\perp(\vx)^\beta}\,d\vx .
\end{equation}
For each boundary patch, write

\begin{equation}\label{eq:line-surface-components}
  s_{i1}=\vm{s}_i\cdot\vm{e}_1,
  \qquad
  s_{i2}=\vm{s}_i\cdot\vm{e}_2,
  \qquad
  s_{i3}=\vm{s}_i\cdot\vm{e}_3 .
\end{equation}
The line-singularity map scales only the transverse coordinates,
\begin{equation}\label{eq:line-map}
  \widetilde{\vm{\varphi}}_i(\rho,t,u)
  = \rho s_{i1}(t,u)\vm{e}_1
  + \rho s_{i2}(t,u)\vm{e}_2
  + s_{i3}(t,u)\vm{e}_3,
  \qquad 0\leq\rho\leq1 .
\end{equation}
The Jacobian determinant of \eqref{eq:line-map} is
\begin{equation}\label{eq:line-map-jacobian}
  \widetilde{J}_i(\rho,t,u)
  = \rho\left(
    s_{i1}\left(\partial_t s_{i2}\partial_u s_{i3}-\partial_u s_{i2}\partial_t s_{i3}\right)
    -s_{i2}\left(\partial_t s_{i1}\partial_u s_{i3}-\partial_u s_{i1}\partial_t s_{i3}\right)
  \right) .
\end{equation}
Equivalently, if $\vm{s}_{i\perp}=(s_{i1},s_{i2})$ and
$r_{i\perp}=\|\vm{s}_{i\perp}\|$, then the distance from the mapped point to the singular line satisfies
\begin{equation}\label{eq:line-distance}
  \operatorname{dist}\left(\widetilde{\vm{\varphi}}_i(\rho,t,u),\{\lambda\vm{e}_3:\lambda\in\Re\}\right)
  =\rho r_{i\perp} .
\end{equation}
Consequently,
\begin{equation}\label{eq:line-transformed-integral}
  I = \sum_i \int_0^1\int_{D_i}
  g\bigl(\widetilde{\vm{\varphi}}_i(\rho,t,u)\bigr)
  \frac{\rho^{1-\beta}}{r_{i\perp}^{\beta}}
  \left(
    s_{i1}\left(\partial_t s_{i2}\partial_u s_{i3}-\partial_u s_{i2}\partial_t s_{i3}\right)
    -s_{i2}\left(\partial_t s_{i1}\partial_u s_{i3}-\partial_u s_{i1}\partial_t s_{i3}\right)
  \right)
  \,dt\,du\,d\rho .
\end{equation}
For the common kernel $r_\perp(\vx)^{-1}$, the factor $\rho$ in the Jacobian cancels the transverse radial singularity
exactly.  For this three-dimensional line case, the transformed radial power is $\rho^{1-\beta}$, reflecting the
two-dimensional transverse measure.  Fractional values of $\beta$ can be handled by Gauss--Jacobi quadrature in $\rho$
or by a generalized transverse scaling $\rho=\xi^\alpha$, which changes the radial power to $\xi^{\alpha(2-\beta)-1}$.
Note, the construction assumes a straight singular line, or a local approximation in which a straight tangent frame is
adequate.  Curved singular lines require local frames and additional geometric terms.

The denominator $r_{i\perp}^\beta$ in \eqref{eq:line-transformed-integral} is the analogue of the boundary-patch factor
in the point-singularity formula.  If the singular line intersects, or nearly intersects, a boundary patch, then
$r_{i\perp}$ can vanish or become small in the patch parameters and the transverse scaling alone is insufficient.  This
is the three-dimensional counterpart of the edge-direction near singularities treated by distance transformations in two
dimensions~\cite{Ma:2002:DTN,Lv:2019:ASD,Chin:2021:SBC}.  The patch should then be subdivided so the intersection is
isolated, or an additional parameter transformation should be applied on the affected patch.  The numerical study in
\sref{sec:ex-line-singularities} tests this straight-line transformation for a centered line and then moves the line
toward a boundary patch to show how a face-level parameter transformation removes the remaining near-singular patch
behavior.

This line-singularity map and the standard SB map can also be viewed as members of a broader family of
scaled-boundary-like transformations.  Such maps use boundary information to define canonical quadrature points by
preserving selected coordinates and scaling the remaining transverse coordinates toward a lower-dimensional reference
set.  For weakly singular functions, this viewpoint is useful when the preserved coordinates can be aligned with the
singular set.  If an affine singular set has dimension $m$ in $\Re^d$, then its transverse codimension is $c=d-m$; the
local kernel $r_\perp^{-\beta}$ is integrable for $0<\beta<c$, and transverse scaling leaves the radial power
$\rho^{c-1-\beta}$.  The point-singularity map corresponds to $m=0$, the straight-line case above corresponds to $m=1$
and $d=3$, and a plane-collapse map in three dimensions would correspond to $m=2$ and $d=3$.  The same
coordinate-preserving viewpoint also includes nonsingular integration rules.  In particular, the planar
Sommariva--Vianello Gauss--Green and product Gauss constructions are related $m=1$, $d=2$ members of the family: they
preserve a boundary coordinate and use boundary information to generate canonical cubature rules without first forming
an area mesh \cite{Sommariva:2007:PGC,Sommariva:2009:GGC}.

\subsection{Practical use}\label{sec:singular-practical}

The preceding formulas, together with the distance-transformation results for near-singular boundary integrals
\cite{Ma:2002:DTN,Lv:2019:ASD,Chin:2021:SBC}, suggest the following strategy.  For point singularities, choose the
volume scaling center at the singular point whenever the oriented swept sectors are admissible and the integrand is
defined on them.  For affine singular sets of positive dimension, align the coordinate system with tangent and
transverse directions and scale only the transverse coordinates.  In both cases, if the singular set intersects or
nearly intersects a boundary patch, the transformation reduces or cancels the radial singularity in the swept coordinate
but does not remove singular behavior that remains in the boundary parametrization.  The affected patch can then be
treated with a local parameter transformation, with subdivision used if needed to isolate the difficult feature.  The
point and line singularity examples in \sref{sec:ex-point-singularities}--\sref{sec:ex-line-singularities} exercise
these choices: the point-source examples include a face-level transfer for a near-boundary source, while the line-source
examples compare the untransformed and transferred boundary-patch rules as a line approaches a boundary patch.

Table~\ref{tab:singular-radial-powers} summarizes the radial powers produced by the transformations.  The exponent is
the power of the scaling coordinate that remains after multiplying the singular kernel by the corresponding Jacobian.
Nonnegative integer exponents are convenient for polynomial exactness with ordinary Gauss rules.  For fractional
exponents, Gauss--Jacobi weighting avoids increasing the degree of the composed integrand; this option is used in the
examples below and was also effective in the planar SBC setting~\cite{Chin:2021:SBC}.

\begin{table}[t]
  \centering
  \caption{Radial powers remaining after scaled boundary transformations for representative weak singularities.}
  \label{tab:singular-radial-powers}
  \begin{tabular}{llll}
    \hline
    Singularity & Setting & Transformation & Remaining radial power \\
    \hline
    point $r^{-\beta}$ & $d$ & standard SB & $\rho^{d-1-\beta}$ \\
    point $r^{-\beta}$ & $d$ & generalized SB, $\rho=\xi^\alpha$ & $\xi^{\alpha(d-\beta)-1}$ \\
    affine set $r_\perp^{-\beta}$ & codim. $c$ & transverse scaling & $\rho^{c-1-\beta}$ \\
    affine set $r_\perp^{-\beta}$ & codim. $c$ & generalized transverse scaling & $\xi^{\alpha(c-\beta)-1}$ \\
    \hline
  \end{tabular}
\end{table}

\section{Numerical results}\label{sec:results}

This section reports numerical studies that verify the smooth and singular cubature constructions introduced above.  The
first examples establish the baseline behavior of the tensor-product scaled-boundary rule on smooth curved solids.  The
remaining examples examine affine polyhedra and nonconvexity, a four-dimensional affine polytope, and weakly singular
integrands.

\subsection{Curved surface-bounded solids}\label{sec:ex-curved-solids}

We first consider two solids whose oriented boundary patches are swept from the origin.  The first is a watertight
quadratic B-spline pillow made from six tensor-product patches.  Let $\vm{h}=(1.35,1.00,0.75)$ denote the half widths of
the underlying rectangular box in the $\vm{e}_1$, $\vm{e}_2$, and $\vm{e}_3$ directions, and set $\delta=0.35$.  Define
the half-axis vectors $\vm{q}_k=h_k\vm{e}_k$.  Each face is specified by an outward face vector $\vm{n}$ and two tangent
half-axis vectors $\vm{a}$ and $\vm{b}$, ordered so that $\vm{a}\times\vm{b}$ points outward.  The six triples are
\begin{equation}
  (\vm{n},\vm{a},\vm{b})\in\{
  (\vm{q}_1,\vm{q}_2,\vm{q}_3),
  (-\vm{q}_1,\vm{q}_3,\vm{q}_2),
  (\vm{q}_2,\vm{q}_3,\vm{q}_1),
  (-\vm{q}_2,\vm{q}_1,\vm{q}_3),
  (\vm{q}_3,\vm{q}_1,\vm{q}_2),
  (-\vm{q}_3,\vm{q}_2,\vm{q}_1)
  \}.
  \label{eq:pillow-face-triples}
\end{equation}
For each triple, the polynomial patch is
\begin{equation}
  \vm{s}(r,t)=\left[1+\delta(1-r^2)(1-t^2)\right]\vm{n}+r\vm{a}+t\vm{b},
  \qquad -1\le r,t\le 1 .
  \label{eq:pillow-surface}
\end{equation}
This polynomial form is the explicit evaluation of the same quadratic B-spline patch used in the computations.
Equivalently, with $r=2u-1$, $t=2v-1$, the open quadratic knot vector $\{0,0,0,1,1,1\}$, and Bernstein control
coordinates $r_i,t_j\in\{-1,0,1\}$, the control points are
\begin{equation}
  \vm{P}_{ij}=\vm{n}+r_i\vm{a}+t_j\vm{b}+4\delta\vm{n}\,\mathbf{1}_{i=1,j=1},
  \qquad i,j=0,1,2 .
  \label{eq:pillow-control-net}
\end{equation}
The second geometry is the quadratic NURBS torus from \cite{Chin:2020:AEM}, with major radius $3$ and minor radius $1$.
The geometries are shown in \fref{fig:ex-curved-solid-geometries}; the black curves show the control nets and the
interior points show a representative fourth-order scaled-boundary rule.  The B-spline pillow tests the polynomial case
in which the patch maps are polynomial, whereas the torus tests a rational surface representation for which polynomial
exactness is not expected.

For each nonzero knot span, we apply an $n$-point Gauss rule in the two surface parameters and an $n$-point Gauss rule
in the radial coordinate $\rho$.  The reported rule size is therefore $n\times n\times n$ per knot-span sector.  Two
smooth integrands are used:
\begin{equation}
  f_1(\vx) = x^2 + 2y^2 - 3z^2,
  \qquad
  f_2(\vx) = \exp(x+y+z).
\end{equation}
Reference values are computed with a higher-order scaled-boundary rule, except for the polynomial torus integral where
the analytic torus moments are used.  Errors are plotted in \fref{fig:ex-curved-solid-errors}, and the first rule orders
that reach selected tolerances are listed in \tref{tab:ex-curved-solid-thresholds}.

\begin{figure}[!htbp]
  \centering
  \begin{subfigure}[t]{0.47\linewidth}
    \centering
    \includegraphics[width=\linewidth]{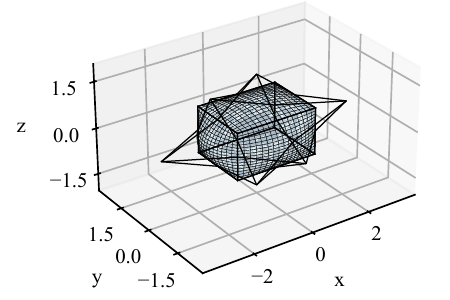}
    \caption{}
    \label{fig:ex-bspline-geometry}
  \end{subfigure}
  \hfill
  \begin{subfigure}[t]{0.47\linewidth}
    \centering
    \includegraphics[width=\linewidth]{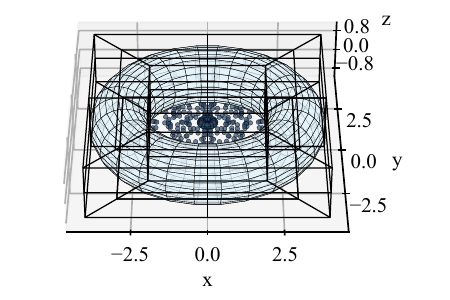}
    \caption{}
    \label{fig:ex-torus-geometry}
  \end{subfigure}
  \caption{Curved surface patches used for the smooth three-dimensional tests.  The black curves show the control nets,
  and the plotted interior points show a representative fourth-order scaled-boundary cubature rule.  (a) Quadratic
  B-spline pillow.  (b) Quadratic NURBS torus.}
  \label{fig:ex-curved-solid-geometries}
\end{figure}

\begin{figure}[!htbp]
  \centering
  \begin{subfigure}[t]{0.47\linewidth}
    \centering
    \includegraphics[width=\linewidth]{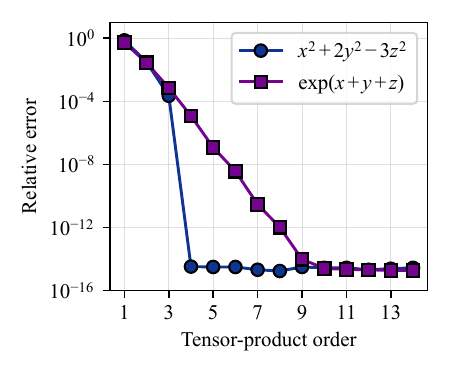}
    \caption{}
    \label{fig:ex-bspline-errors}
  \end{subfigure}
  \hfill
  \begin{subfigure}[t]{0.47\linewidth}
    \centering
    \includegraphics[width=\linewidth]{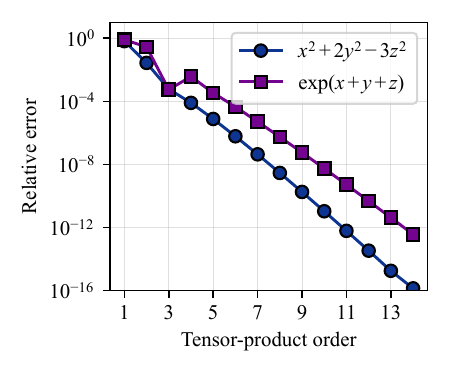}
    \caption{}
    \label{fig:ex-torus-errors}
  \end{subfigure}
  \caption{Convergence of the tensor-product scaled-boundary rule for smooth integrands over curved solids.  The
  B-spline pillow is polynomial and reaches machine precision for the polynomial integrand at modest order.  The NURBS
  torus exhibits rapid convergence but is not polynomial-exact in the tensor-product coordinates.  (a) Quadratic
  B-spline pillow.  (b) Quadratic NURBS torus.}
  \label{fig:ex-curved-solid-errors}
\end{figure}

\begin{table}[!htbp]
  \centering
  \caption{First tensor-product rule orders that reach selected error tolerances for the curved-solid tests.  Each entry
  denotes an $n\times n\times n$ rule on each nonzero knot-span sector.}
  \label{tab:ex-curved-solid-thresholds}
  \begin{tabular}{llccc}
    \hline
    & & \multicolumn{3}{c}{Relative error} \\
    \cline{3-5}
    Geometry & Integrand & $<10^{-4}$ & $<10^{-8}$ & $<10^{-12}$ \\
    \hline
    B-spline pillow & $x^2+2y^2-3z^2$ & $4\times4\times4$ & $4\times4\times4$ & $4\times4\times4$ \\
    B-spline pillow & $\exp(x+y+z)$ & $4\times4\times4$ & $6\times6\times6$ & $9\times9\times9$ \\
    NURBS torus & $x^2+2y^2-3z^2$ & $4\times4\times4$ & $8\times8\times8$ & $11\times11\times11$ \\
    NURBS torus & $\exp(x+y+z)$ & $6\times6\times6$ & $10\times10\times10$ & $14\times14\times14$ \\
    \hline
  \end{tabular}
\end{table}

The B-spline pillow polynomial case reaches roundoff once the rule is sufficiently high to integrate the polynomial
mapped integrand.  The exponential integrand is not polynomial, but the error decreases rapidly as the tensor-product
order is increased.  On the torus, the rational NURBS map destroys polynomial exactness even for $f_1$, so the
polynomial and exponential tests both show convergence rather than exactness at a fixed low order.  These results are
consistent with the discussion in \sref{sec:sbc}: polynomial exactness is available for polynomial patch maps and
polynomial integrands, whereas rational patch maps are handled by ordinary smooth quadrature convergence.

\FloatBarrier

\subsection{Affine polyhedra and nonconvexity}\label{sec:ex-affine-polyhedra}

We next test the affine specialization of the rule on polyhedra with polygonal and triangular faces.  The examples in
\fref{fig:ex-polyhedra-rules} include the Szilassi polyhedron, a truncated hexagonal trapezohedron, the Echidnahedron,
and the dodecahedron-small triambic icosahedron compound; these polyhedral shapes are taken from \cite{Chin:2020:AEM}.
The plotted points are produced by a $4\times4\times3$ rule on each tetrahedral sector generated by the nested map in
\sref{sec:sbc}.  This rule integrates polynomials of degree at most five on each sector because the Jacobian contributes
the factor $\rho^2\eta$ and the affine coordinate map leaves a polynomial integrand polynomial in $(\rho,\eta,t)$.

\begin{figure}[!htbp]
  \centering
  \begin{subfigure}[t]{0.47\linewidth}
    \centering
    \includegraphics[width=\linewidth]{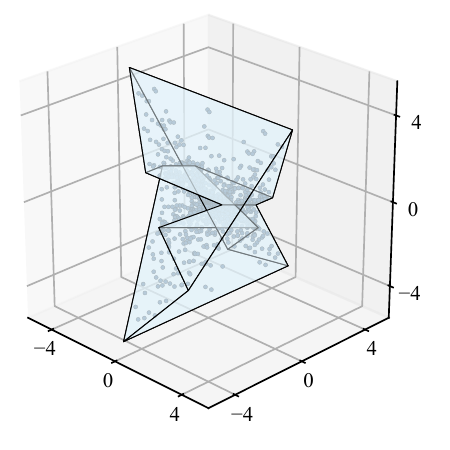}
    \caption{}
    \label{fig:ex-szi-rule}
  \end{subfigure}
  \hfill
  \begin{subfigure}[t]{0.47\linewidth}
    \centering
    \includegraphics[width=\linewidth]{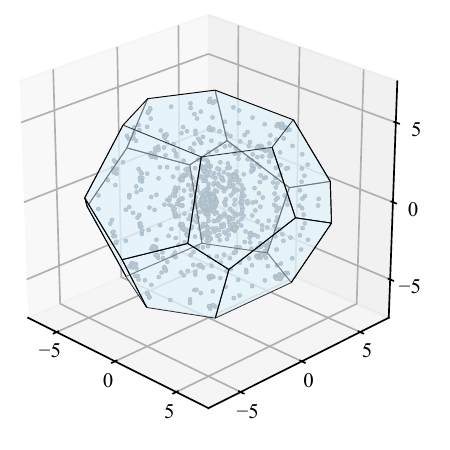}
    \caption{}
    \label{fig:ex-tht-rule}
  \end{subfigure}
  \begin{subfigure}[t]{0.47\linewidth}
    \centering
    \includegraphics[width=\linewidth]{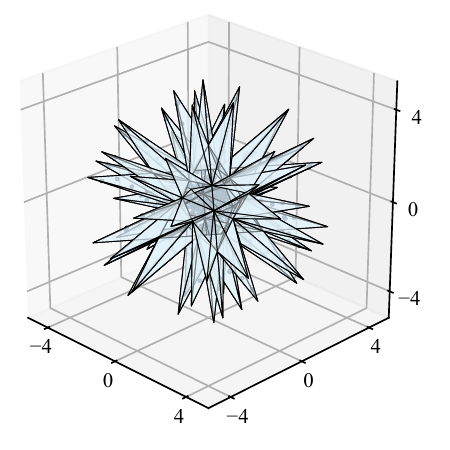}
    \caption{}
    \label{fig:ex-echid-rule}
  \end{subfigure}
  \hfill
  \begin{subfigure}[t]{0.47\linewidth}
    \centering
    \includegraphics[width=\linewidth]{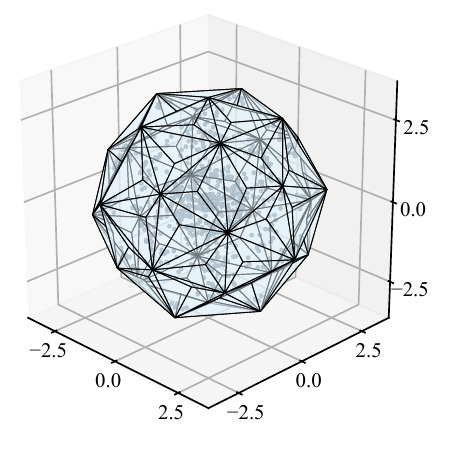}
    \caption{}
    \label{fig:ex-dstic-rule}
  \end{subfigure}
  \caption{Affine-polyhedron examples with representative $4\times4\times3$ scaled-boundary cubature points.  Each
  polygonal face is triangulated by the nested face map, and each face triangle is then scaled from the chosen center.
  (a) Szilassi polyhedron.  (b) Truncated hexagonal trapezohedron.  (c) Echidnahedron.  (d) Dodecahedron-small triambic
  icosahedron compound.}
  \label{fig:ex-polyhedra-rules}
\end{figure}

To verify exactness, we integrate every monomial $x^a y^b z^c$ with $a+b+c\le 5$.  Reference values are computed
analytically on the same oriented tetrahedral sector decomposition by expanding each affine coordinate map in
barycentric coordinates.  The largest error over all tested monomials is shown in \tref{tab:ex-polyhedron-exactness}.
The errors are at roundoff level for all four geometries, including the nonconvex Szilassi polyhedron.  This confirms
that the nested affine formula achieves the expected polynomial exactness when the boundary orientation is consistent.

\begin{table}[!htbp]
  \centering
  \caption{Polynomial exactness test for the $4\times4\times3$ affine-polyhedron rule.  The table reports the largest
  error over all monomials of total degree at most five.}
  \label{tab:ex-polyhedron-exactness}
  \begin{tabular}{lccc}
    \hline
    Geometry & Cubature points & Max. absolute error & Max. relative error \\
    \hline
    Szilassi polyhedron & 1344 & $1.05\times10^{-11}$ & $2.83\times10^{-15}$ \\
    Truncated hexagonal trapezohedron & 2112 & $1.46\times10^{-10}$ & $1.37\times10^{-13}$ \\
    Echidnahedron & 8640 & $3.41\times10^{-13}$ & $2.18\times10^{-14}$ \\
    Dodecahedron-small triambic icosahedron & 8640 & $1.82\times10^{-12}$ & $2.28\times10^{-13}$ \\
    \hline
  \end{tabular}
\end{table}

The same sector formula can produce signed weights when the scaling center is not star-convex with respect to the
oriented boundary.  To quantify this effect, \tref{tab:ex-polyhedron-weight-metrics} reports $\sum_q |w_q|/|\sum_q w_q|$
and the number of negative weights for the origin-centered rules shown in \fref{fig:ex-polyhedra-rules}.  The truncated
hexagonal trapezohedron, Echidnahedron, and dodecahedron-small triambic icosahedron have positive weights for this
center.  The Szilassi polyhedron, however, has both positive and negative sectors, and the sum of absolute weights is
about $1.55$ times the net oriented volume.  The exactness test is still satisfied, but the signed weights indicate the
cancellation that can arise in nonconvex or non-star-convex settings.

\begin{table}[!htbp]
  \centering
  \caption{Signed-weight diagnostics for the origin-centered $4\times4\times3$ affine-polyhedron rules.}
  \label{tab:ex-polyhedron-weight-metrics}
  \begin{tabular}{lccc}
    \hline
    Geometry & $\sum_q |w_q|/|\sum_q w_q|$ & Positive weights & Negative weights \\
    \hline
    Szilassi polyhedron & 1.546 & 672 & 672 \\
    Truncated hexagonal trapezohedron & 1.000 & 2112 & 0 \\
    Echidnahedron & 1.000 & 8640 & 0 \\
    Dodecahedron-small triambic icosahedron & 1.000 & 8640 & 0 \\
    \hline
  \end{tabular}
\end{table}

\FloatBarrier

\subsection{Four-dimensional affine polytope}\label{sec:ex-affine-tesseract}

The preceding examples focus on three-dimensional regions.  To exercise the higher-dimensional recursive construction
more directly, we also integrate over a rotated and sheared affine image of the tesseract $[-1,1]^4$.  The map has the
form
\begin{equation}
  \vx = \vm{A} \hat{\vx} + \vm{b},
  \qquad \hat{\vx}\in[-1,1]^4,
\end{equation}
with
\begin{equation}
  \vm{A}=\begin{pmatrix}
     1.03736796 & -0.19598887 & -0.27921549 & -0.06768860 \\
     0.49957008 &  0.90454137 & -0.13446309 & -0.03259711 \\
     0.31769041 &  0.05956695 &  0.99179131 &  0.51405037 \\
    -0.11562985 & -0.02168060 & -0.36098252 &  0.66424318
  \end{pmatrix},
  \qquad
  \vm{b}=\begin{pmatrix}0.20\\-0.10\\0.15\\0.05\end{pmatrix}.
  \label{eq:tesseract-affine-map}
\end{equation}
Equivalently, the matrix is generated as
\begin{equation}
  \vm{A}=\vm{R}_{12}(\pi/7)\,\vm{R}_{34}(-\pi/9)\,\vm{R}_{13}(\pi/11)\,\vm{H}\,\vm{D},
\end{equation}
where $\vm{R}_{ij}(\theta)$ denotes a rotation by angle $\theta$ in the $(i,j)$ coordinate plane,
\begin{equation}
  \vm{D}=\operatorname{diag}(1.20,0.90,1.10,0.80),
  \qquad
  \vm{H}=\begin{pmatrix}
    1 & 1/4 & 0 & 0 \\
    0 & 1 & 0 & 0 \\
    0 & 0 & 1 & 1/3 \\
    0 & 0 & 0 & 1
  \end{pmatrix}.
\end{equation}
Thus the reference tesseract is first anisotropically scaled, then sheared in two coordinate directions, and finally
rotated in the $(1,3)$, $(3,4)$, and $(1,2)$ coordinate planes.  The rotations are not required by the construction;
they are included only to make the affine image oblique to the coordinate axes while preserving a simple determinant.
The nested SBC rule uses the image of the origin as the scaling center, sweeps to each cubical three-dimensional facet,
triangulates each facet into six tetrahedra, and applies the nested affine-sector map on the resulting four-simplices.
Tensor-product Gauss quadrature on the reference tesseract is used as an independent verification rule.

For the chosen affine map, the exact volume is $16|\det\vm{A}|=15.2064$.  On each four-simplex sector, the volume
Jacobian is proportional to $a^3b^2c$ in the nested coordinates $(a,b,c,d)$, so a $2\times2\times1\times1$ Gauss rule is
already exact for volume.  This minimal nested SBC rule evaluates to $15.206400000000002$, while a sixteenth-order
tensor-product rule returns $15.206400000000002$.  Thus both absolute volume errors are $3.6\times10^{-15}$.  A
fifth-order nested rule is then used for the moment study; it integrates all monomial moments through total degree five
to roundoff when compared with the tensor-product verification values, as shown in \tref{tab:ex-tesseract-moments}.

\begin{table}[!htbp]
  \centering
  \caption{Maximum moment errors for the rotated and sheared affine tesseract.  The nested SBC rule uses order five in
  each tensor-product coordinate and is compared with a sixteenth-order tensor-product rule on the reference tesseract.}
  \label{tab:ex-tesseract-moments}
  \begin{tabular}{cccc}
    \hline
    Total degree & Number of moments & Max absolute error & Max relative error \\
    \hline
    0 & 1 & $2.13\times10^{-14}$ & $1.40\times10^{-15}$ \\
    1 & 4 & $3.11\times10^{-15}$ & $1.36\times10^{-15}$ \\
    2 & 10 & $3.55\times10^{-15}$ & $1.33\times10^{-15}$ \\
    3 & 20 & $6.22\times10^{-15}$ & $2.19\times10^{-15}$ \\
    4 & 35 & $1.07\times10^{-14}$ & $2.27\times10^{-15}$ \\
    5 & 56 & $1.15\times10^{-14}$ & $2.33\times10^{-15}$ \\
    \hline
  \end{tabular}
\end{table}

\begin{figure}[!htbp]
  \centering
  \begin{subfigure}[t]{0.48\linewidth}
    \centering
    \includegraphics[width=\linewidth]{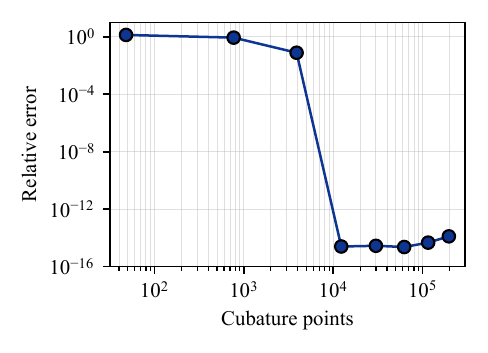}
    \caption{}
    \label{fig:ex-tesseract-moment-convergence}
  \end{subfigure}
  \hfill
  \begin{subfigure}[t]{0.48\linewidth}
    \centering
    \includegraphics[width=\linewidth]{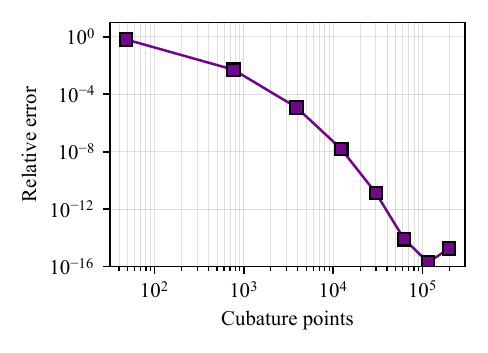}
    \caption{}
    \label{fig:ex-tesseract-exponential-convergence}
  \end{subfigure}
  \caption{Convergence of the nested SBC rule on the rotated and sheared affine tesseract, using tensor-product Gauss
  quadrature on the reference tesseract as the verification rule.  (a) Maximum error over monomial moments of total
  degree at most five.  (b) Error for $\exp(\vm{c}\cdot\vx)$ with a fixed coefficient vector $\vm{c}$.}
  \label{fig:ex-tesseract-convergence}
\end{figure}

The convergence in \fref{fig:ex-tesseract-convergence} is consistent with the polynomial exactness of the nested affine
map.  The maximum error over all moments of total degree at most five reaches roundoff by order four.  For the smooth
nonpolynomial test $\exp(\vm{c}\cdot\vx)$, the relative error decreases from $5.1\times10^{-3}$ at order two to
$1.5\times10^{-8}$ at order four, $1.3\times10^{-11}$ at order five, and roundoff by order seven.  This example verifies
that the recursive affine-polytope construction used in three dimensions extends directly to a non-simplex polytope in
four dimensions.

\FloatBarrier

\subsection{Point singularities}\label{sec:ex-point-singularities}

We now separate the point-source tests into three regimes.  The first tests ordinary scaled-boundary cancellation for
integer powers, the second tests generalized radial scaling for fractional powers, and the third tests the residual
near-boundary factor that remains on a face after the volume scaling has been applied.

\subsubsection{Integer powers away from the boundary}\label{sec:ex-point-integer}

The basic cancellation test uses
\begin{equation}
  I_\beta = \int_\Omega \frac{1}{\|\vx\|^\beta}\,d\vx,
  \qquad \beta\in\{1,2\},
\end{equation}
with the singular point used as the scaling center, $\vx_c=\vx_0=\vm{0}$.  The point is well separated from each boundary
patch of the truncated hexagonal trapezohedron (THT) and the B-spline pillow.  Therefore the only singular factor exposed
by the scaled-boundary map is the radial power $\rho^{2-\beta}$.  For $\beta=1$ this is the polynomial factor $\rho$, and
for $\beta=2$ it is the constant factor $1$.  Ordinary Gauss quadrature in the radial coordinate is therefore sufficient;
generalized scaling or Gauss--Jacobi weighting would give the same radial rule for these two cases.
Reference values are computed by integrating the radial factor analytically, giving $1/(3-\beta)$, and multiplying by a
high-order surface quadrature approximation of the remaining smooth boundary integral.

To separate the radial requirement from the cost of resolving the boundary patches, we next fix the radial rule at
$n_\rho=1$ and vary only the surface order.  The reference values use a high-order surface rule and the exact radial
factor $1/(3-\beta)$.  Since one Gauss point integrates both remaining radial factors exactly, the convergence in
\fref{fig:ex-point-integer-surface} is controlled by the surface quadrature alone.  The THT reaches the $10^{-8}$ and
$10^{-12}$ tolerances at surface orders 8 and 13 for $\beta=1$, and 9 and 14 for $\beta=2$.  The B-spline pillow reaches
$10^{-8}$ at orders 11 and 12 and $10^{-12}$ at orders 16 and 17 for $\beta=1$ and $\beta=2$, respectively.  The B-spline
panel is extended to surface order 24, where the relative errors are about $1.4\times10^{-16}$ and $6.2\times10^{-16}$
for the two kernels.  The higher order for the B-spline pillow is still inexpensive relative to the THT rule: the THT
boundary decomposition produces 44 tetrahedral sectors, so a $16\times16$ surface rule is applied 44 times, whereas the
B-spline case uses six tensor-product patch rules.

\begin{figure}[!htbp]
  \centering
  \begin{subfigure}[t]{0.48\linewidth}
    \centering
    \includegraphics[width=\linewidth]{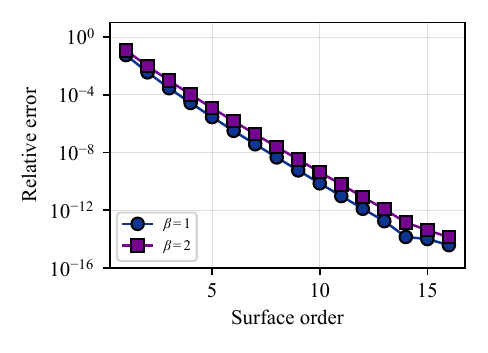}
    \caption{}
    \label{fig:ex-point-integer-surface-tht}
  \end{subfigure}
  \hfill
  \begin{subfigure}[t]{0.48\linewidth}
    \centering
    \includegraphics[width=\linewidth]{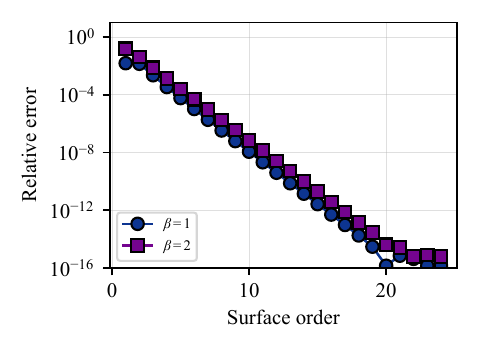}
    \caption{}
    \label{fig:ex-point-integer-surface-bspline}
  \end{subfigure}
  \caption{Surface-order convergence for integer point-singular kernels with the radial order fixed at $n_\rho=1$.  The
  single radial point is sufficient for the $\beta=1$ and $\beta=2$ kernels, so the remaining error comes from the
  boundary-patch integration.  The B-spline panel is run to surface order 24 to show convergence to near-roundoff error.
  (a) Truncated hexagonal trapezohedron.  (b) B-spline pillow.}
  \label{fig:ex-point-integer-surface}
\end{figure}

\begin{figure}[!htbp]
  \centering
  \begin{subfigure}[t]{0.48\linewidth}
    \centering
    \includegraphics[width=\linewidth]{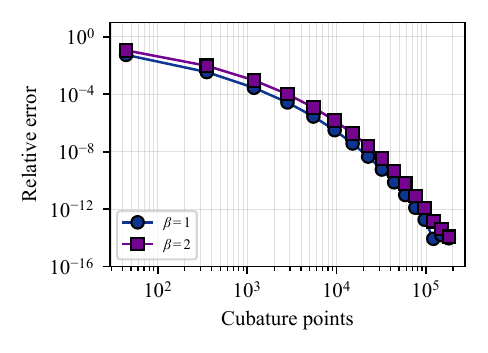}
    \caption{}
    \label{fig:ex-point-integer-standard-tht}
  \end{subfigure}
  \hfill
  \begin{subfigure}[t]{0.48\linewidth}
    \centering
    \includegraphics[width=\linewidth]{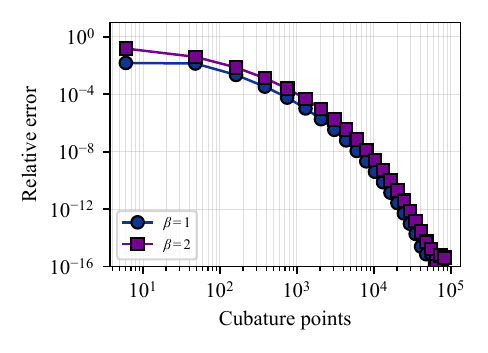}
    \caption{}
    \label{fig:ex-point-integer-standard-bspline}
  \end{subfigure}
  \caption{Standard scaled-boundary convergence for integer point-singular kernels with the singularity far from the
  boundary patches.  For $\beta=1$ and $\beta=2$, the remaining radial powers are polynomial, so no generalized radial
  treatment is needed.  (a) Truncated hexagonal trapezohedron.  (b) B-spline pillow.}
  \label{fig:ex-point-integer-standard}
\end{figure}

The total-cubature plot in \fref{fig:ex-point-integer-standard} shows the same trend with cost measured by the number of
points rather than by the surface order.  The THT panel uses orders through 16 on each tetrahedral sector, while the
B-spline panel is extended through surface order 24 on each of the six pillow patches.  The comparison confirms that,
for these integer kernels away from the boundary, the radial singularity has been removed by the SB map and the observed
convergence is
governed by the boundary quadrature.

\subsubsection{Fractional powers away from the boundary}\label{sec:ex-point-fractional}

The generalized scaled-boundary map is more relevant for fractional powers.  We next use
$\beta\in\{0.5,1.5,2.5\}$ on the same two geometries, again with the singular point far from the boundary.  The standard
map leaves a non-polynomial radial factor $\rho^{2-\beta}$.  The generalized map sets $\rho=\xi^\alpha$ and chooses the
smallest integer $\alpha$ such that $\alpha(3-\beta)-1$ is a nonnegative integer; for these half-integer examples,
$\alpha=2$.  We compare this with a Gauss--Jacobi radial rule that keeps the standard map but integrates the weight
$\rho^{2-\beta}$ directly.

As in the integer-power study, the homogeneity of the kernel lets us separate the radial and surface contributions.
The exact radial factor is again $1/(3-\beta)$.  Generalized SB makes the transformed radial factor polynomial for these
half-integer powers, while Gauss--Jacobi integrates the original radial weight directly; in either case the remaining
convergence is governed by the boundary-patch quadrature.  To show this, \fref{fig:ex-point-fractional-surface} fixes
the radial contribution analytically and varies only the surface order.

\begin{figure}[!htbp]
  \centering
  \begin{subfigure}[t]{0.48\linewidth}
    \centering
    \includegraphics[width=\linewidth]{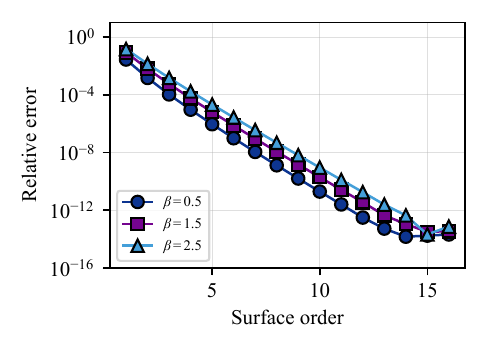}
    \caption{}
    \label{fig:ex-tht-point-fractional-surface-errors}
  \end{subfigure}
  \hfill
  \begin{subfigure}[t]{0.48\linewidth}
    \centering
    \includegraphics[width=\linewidth]{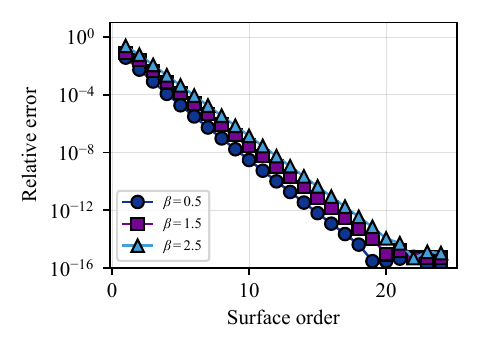}
    \caption{}
    \label{fig:ex-bspline-point-fractional-surface-errors}
  \end{subfigure}
  \caption{Surface-order convergence for fractional point-singular kernels after the homogeneous radial factor is
  integrated exactly.  The THT panel uses surface orders through 16 on each tetrahedral sector, while the B-spline panel
  uses orders through 24 on each of the six tensor-product pillow patches.  (a) Truncated hexagonal trapezohedron.
  (b) B-spline pillow.}
  \label{fig:ex-point-fractional-surface}
\end{figure}

The surface-order results reach the same thresholds as the generalized and Gauss--Jacobi curves in the total-cost plot.
On the THT, the $10^{-8}$ thresholds occur at orders 8, 9, and 9 for $\beta=0.5,1.5,2.5$, respectively, and the
$10^{-12}$ thresholds occur at orders 12, 13, and 14.  On the B-spline pillow, the $10^{-8}$ thresholds occur at orders
10, 11, and 12, and the $10^{-12}$ thresholds occur at orders 15, 17, and 18.

\begin{figure}[!htbp]
  \centering
  \begin{subfigure}[t]{\linewidth}
    \centering
    \makebox[\linewidth][c]{\includegraphics{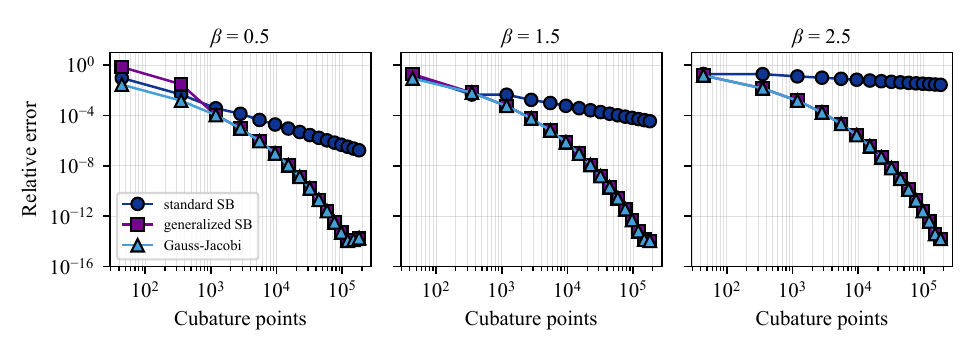}}
    \caption{}
    \label{fig:ex-tht-point-fractional-errors}
  \end{subfigure}
  \par\medskip
  \begin{subfigure}[t]{\linewidth}
    \centering
    \makebox[\linewidth][c]{\includegraphics{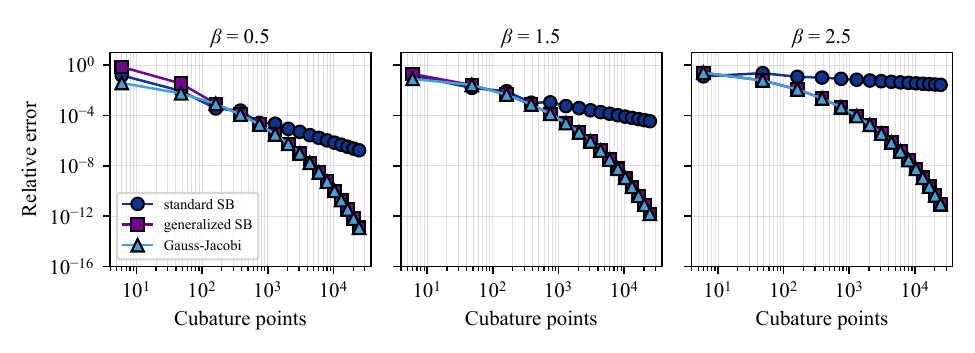}}
    \caption{}
    \label{fig:ex-bspline-point-fractional-errors}
  \end{subfigure}
  \caption{Convergence for fractional point-singular kernels $\|\vx\|^{-\beta}$ with the singular point used as the
  scaling center.  The standard SB rule must approximate a fractional radial power, while generalized SB and
  Gauss--Jacobi quadrature remove or weight that factor directly.  (a) Truncated hexagonal trapezohedron.  (b) B-spline
  pillow.}
  \label{fig:ex-point-fractional-errors}
\end{figure}

The fractional-power results in \fref{fig:ex-point-fractional-errors} isolate the benefit of the generalized transform.
Unlike the surface-order diagnostic in \fref{fig:ex-point-fractional-surface}, these curves use full tensor-product
rules: the same order is used in the surface coordinates and in the radial coordinate, and the horizontal axis reports
the resulting total number of cubature points.  For all three half-integer powers, the standard SB rule does not reach
$10^{-8}$ by order 16.  On the THT, both generalized SB and Gauss--Jacobi reach $10^{-8}$ at orders 8, 9, and 9 for
$\beta=0.5,1.5,2.5$, respectively, and $10^{-12}$ at orders 12, 13, and 14.  On the B-spline pillow, the corresponding
$10^{-8}$ thresholds are orders 10, 11, and 12; the surface-order diagnostic shows that the corresponding $10^{-12}$
thresholds occur at surface orders 15, 17, and 18.

\subsubsection{Near-boundary point sources}\label{sec:ex-point-near-boundary}

The preceding tests remove the radial singularity, after which the remaining error is governed by the boundary-patch
integral.  Even when the source is not asymptotically close to the boundary, the surface integral contribution can still
require many quadrature points.  We now test a more localized version of the same effect using a single tetrahedral
sector.  As shown in \fref{fig:ex-near-face-geometry}, the point source is also the volume scaling center $\vx_0$ at one
vertex of the tetrahedron.  One tetrahedron edge is aligned with the segment from $\vx_0$ to the face scaling center, so
the singularity-aligned decomposition consists of a single tetrahedral sector.  The face scaling center is a vertex of
the opposite face, the two incident face edges have fixed length $L=2$, and the opposite face edge is placed a distance
$d$ from this point.  Thus, as $d\to0$, both the volume face and the opposite face edge approach the singularity, while
the perimeter of the face approaches $4L$. The outer SB map cancels the $1/r$ singularity in the volume radial
coordinate, but the remaining boundary integrand $\|\vm{s}-\vx_c\|^{-1}$ on the opposite face becomes nearly singular
near the face scaling center and the nearby edge.

The face scaling center is chosen as the closest point on the triangular face to the source.  The local SB map on the
face uses a face-radial coordinate $\eta$ from this center to the opposite edge and an edge coordinate $\tau$ along that
edge.  The $t$-transfer is applied only to the edge coordinate $\tau$.  In the face-radial direction, the local SB surface Jacobian
already supplies the factor $\eta$; in the limiting case $d=0$, this exactly cancels the $1/\eta$ behavior of the
remaining boundary factor.  On the opposite edge, the residual edge-direction distance has the limiting form
$r(\tau)=\sqrt{\ell^2+\tau^2}$, where $\ell$ is the perpendicular distance from the face scaling center to the edge.  We
therefore compare ordinary Gauss quadrature in $\tau$ with the $t$-transfer used for near-singular edge factors in the
planar SBC scheme~\cite{Chin:2021:SBC}:
\begin{equation}
  \tilde{\tau}=\log\left(\tau+\sqrt{\ell^2+\tau^2}\right),
  \qquad
  \tau=\frac{1}{2}e^{-\tilde{\tau}}\left(e^{2\tilde{\tau}}-\ell^2\right).
\end{equation}

\begin{figure}[!htbp]
  \centering
  \begin{subfigure}[t]{0.49\linewidth}
    \centering
    \makebox[\linewidth][c]{\includegraphics{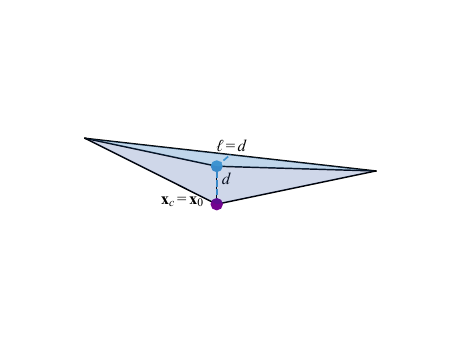}}
    \caption{}
    \label{fig:ex-near-face-geometry-3d}
  \end{subfigure}
  \hfill
  \begin{subfigure}[t]{0.49\linewidth}
    \centering
    \makebox[\linewidth][c]{\includegraphics{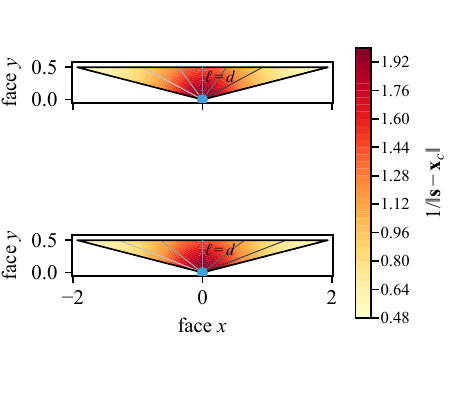}}
    \caption{}
    \label{fig:ex-near-face-heatmap}
  \end{subfigure}
  \caption{Near-boundary point-source geometry.  The 3D panel shows the single tetrahedral sector and the equal distances
  $d$ and $\ell=d$.  The face panel shows the remaining boundary factor $\|\vm{s}-\vx_c\|^{-1}$ and compares isocontours
  of the transferred and untransformed edge coordinates; darker contour lines correspond to larger coordinate values.
  (a) Tetrahedral sector geometry.  (b) Face factor with $\tilde{\tau}$ contours on top and untransformed $\tau$ contours
  below.}
  \label{fig:ex-near-face-geometry}
\end{figure}

\begin{figure}[!htbp]
  \centering
  \begin{subfigure}[t]{0.48\linewidth}
    \centering
    \makebox[\linewidth][c]{\includegraphics{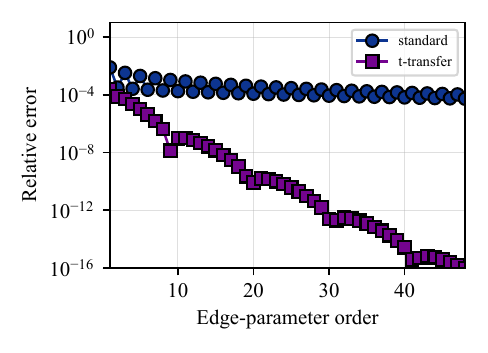}}
    \caption{}
    \label{fig:ex-near-face-convergence-d1em2}
  \end{subfigure}
  \hfill
  \begin{subfigure}[t]{0.48\linewidth}
    \centering
    \makebox[\linewidth][c]{\includegraphics{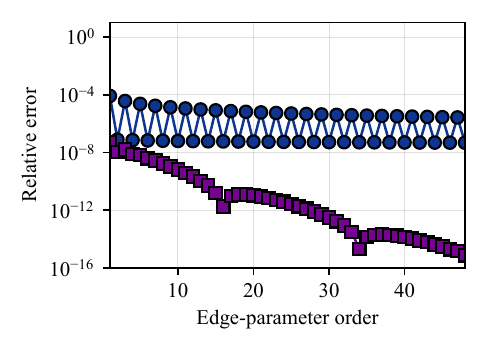}}
    \caption{}
    \label{fig:ex-near-face-convergence-d1em4}
  \end{subfigure}
  \caption{Near-boundary point-source convergence with increasing edge-direction order.  The face-radial coordinate
  is integrated with 720 Gauss points, so the plotted errors primarily reflect edge-direction quadrature.  The
  $t$-transfer reduces the near-singularity left on the face when the opposite edge approaches the source.  (a) $d=10^{-2}$.  (b) $d=10^{-4}$.}
  \label{fig:ex-near-face-convergence}
\end{figure}

\begin{figure}[!htbp]
  \centering
  \begin{subfigure}[t]{0.48\linewidth}
    \centering
    \makebox[\linewidth][c]{\includegraphics{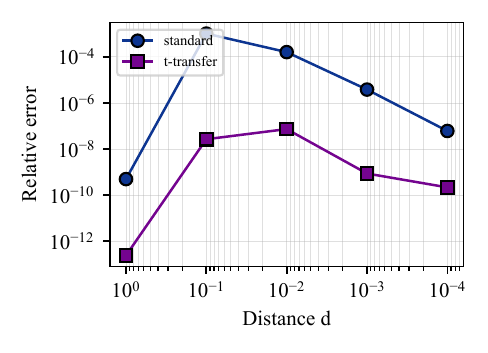}}
    \caption{}
    \label{fig:ex-near-face-distance-order12}
  \end{subfigure}
  \hfill
  \begin{subfigure}[t]{0.48\linewidth}
    \centering
    \makebox[\linewidth][c]{\includegraphics{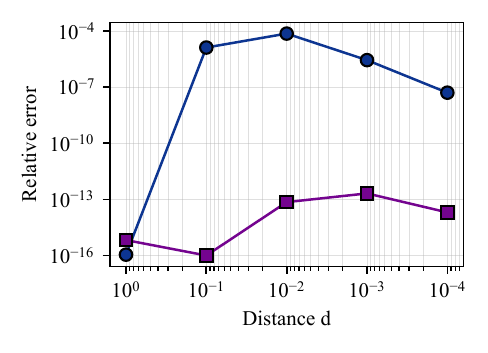}}
    \caption{}
    \label{fig:ex-near-face-distance-order36}
  \end{subfigure}
  \caption{Near-boundary point-source error as the source approaches the face.  The separation between the standard
  edge rule and the $t$-transfer increases when a sufficiently high edge-direction order is used.  (a) Edge-direction
  order 12.  (b) Edge-direction order 36.}
  \label{fig:ex-near-face-distance}
\end{figure}

The order-convergence and distance-sweep results in \fref{fig:ex-near-face-convergence} and
\fref{fig:ex-near-face-distance}, respectively, confirm that the volume radial cancellation and the face-level transfer address
different factors.  At $d=1$,
where the source is not close to the face, the standard rule reaches the $10^{-8}$ and $10^{-12}$ tolerances at
orders 11 and 16, while the $t$-transfer reaches them at orders 7 and 11.  At $d=10^{-1}$, the standard edge rule does
not reach $10^{-8}$ over the tested range of edge orders, whereas the $t$-transfer reaches the $10^{-8}$ and $10^{-12}$
tolerances at orders 13 and 25.  At $d=10^{-2}$ and $d=10^{-4}$, the standard rule again does not reach $10^{-8}$, while
the $t$-transfer reaches $10^{-8}$ at orders 16 and 4 and $10^{-12}$ at orders 30 and 28, respectively.  This shows that
the $t$-transfer is useful when true point singularities leave large, localized boundary contributions that would
otherwise be poorly resolved.

\FloatBarrier

\subsection{Line singularities}\label{sec:ex-line-singularities}

We next test the transverse scaling map from \sref{sec:line}.  The test integrals are
\begin{equation}
  I_\beta = \int_{[-1,1]^3}\frac{1}{(x^2+y^2)^{\beta/2}}\,d\vx,
  \qquad \beta\in\{0.5,1,1.5\},
\end{equation}
where the singular line is the $z$-axis.

\subsubsection{Centered line}\label{sec:ex-line-centered}

The centered-line test uses the cube because the vertical side faces are star-convex with respect to the line, while the
top and bottom faces have zero measure under the transverse sweep.  We compare four rules: ordinary volume SB cubature
with scaling center on the line, the transverse scaling map with ordinary Gauss quadrature in the transverse coordinate,
generalized transverse scaling, and the same transverse scaling map with Gauss--Jacobi quadrature for the transverse
factor $\rho^{1-\beta}$.  Reference values are computed with a high-order Gauss--Jacobi transverse rule.

\begin{figure}[!htbp]
  \centering
  \makebox[\linewidth][c]{\includegraphics{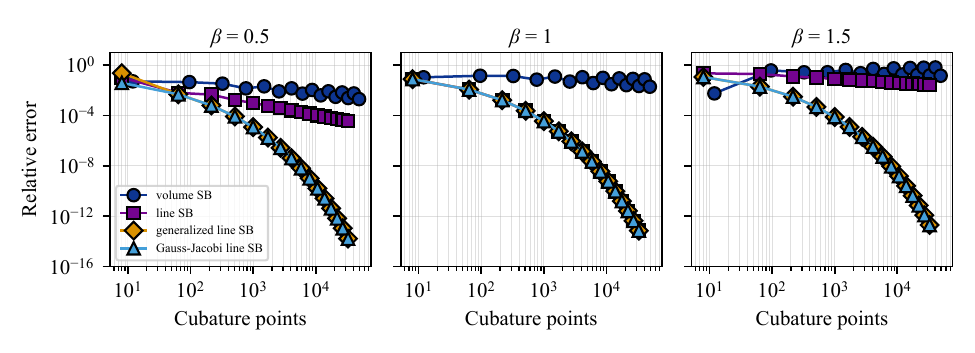}}
  \caption{Convergence for centered line-singular kernels on the cube.  The transverse scaling map isolates the line singularity in the transverse coordinate.  For fractional transverse powers, generalized transverse scaling or Gauss--Jacobi quadrature
  removes the remaining noninteger radial factor.}
  \label{fig:ex-line-singularity-convergence}
\end{figure}

The centered-line convergence in \fref{fig:ex-line-singularity-convergence} follows the radial powers in
\tref{tab:singular-radial-powers}.  The ordinary transverse scaling map leaves the transverse factor $\rho^{1-\beta}$; this is a
constant for $\beta=1$, but fractional for $\beta=0.5$ and singular for $\beta=1.5$.  Generalized transverse scaling,
with $\rho=\xi^\alpha$, and the Gauss--Jacobi rule remove or weight that factor directly, giving the fastest convergence
for the fractional cases.
The ordinary volume SB rule converges slowly here because the top and bottom faces intersect the singular line in their
face parameters, so the point-based sweep does not isolate the line singularity as cleanly as the transverse scaling map.
The first orders that reach the selected error tolerances are summarized in
\tref{tab:ex-line-singularity-thresholds}.

\begin{table}[!htbp]
  \centering
  \caption{First tensor-product orders reaching selected relative-error tolerances for the centered line-singularity
  tests.  A dash indicates that the tolerance was not reached by order 16.}
  \label{tab:ex-line-singularity-thresholds}
  \begin{tabular}{llcc}
    \hline
    & & \multicolumn{2}{c}{Relative error} \\
    \cline{3-4}
    Kernel & Treatment & $<10^{-8}$ & $<10^{-12}$ \\
    \hline
    $\beta=0.5$ & volume SB & -- & -- \\
    $\beta=0.5$ & transverse scaling & -- & -- \\
    $\beta=0.5$ & gen. transverse scaling / Gauss--Jacobi & 9 & 14 \\
    $\beta=1$ & volume SB & -- & -- \\
    $\beta=1$ & transverse scaling / gen. transverse scaling / Gauss--Jacobi & 10 & 15 \\
    $\beta=1.5$ & volume SB & -- & -- \\
    $\beta=1.5$ & transverse scaling & -- & -- \\
    $\beta=1.5$ & gen. transverse scaling / Gauss--Jacobi & 10 & 16 \\
    \hline
  \end{tabular}
\end{table}

\subsubsection{Near-boundary line}\label{sec:ex-line-near-boundary}

The second line test moves the singular line to $(1-d,0,z)$, leaving it inside the cube but bringing it close to the side
face $x=1$.  The transverse scaling map still removes the transverse swept-coordinate singularity, but on that side face the
remaining boundary factor is
\begin{equation}
  r_{\perp}^{-\beta}=\left(d^2+y^2\right)^{-\beta/2},
\end{equation}
where $y$ is measured along the face from the closest point to the line.  This is the same residual patch-parameter
near singularity identified in \sref{sec:line}.  For the $\beta=1$ kernel, we therefore apply the one-dimensional
$t$-transfer to the affected face coordinate,
\begin{equation}
  \tilde{y}=\log\left(y+\sqrt{d^2+y^2}\right),
  \qquad
  y=\frac{1}{2}e^{-\tilde{y}}\left(e^{2\tilde{y}}-d^2\right).
\end{equation}
Since $dy/d\tilde{y}=\sqrt{d^2+y^2}$, this transfer cancels the $1/\sqrt{d^2+y^2}$ factor on the near face.  The
transformation is applied only to the side face $x=1$; all other boundary faces use ordinary transverse-scaling surface quadrature.

\begin{figure}[!htbp]
  \centering
  \begin{subfigure}[t]{\linewidth}
    \centering
    \makebox[\linewidth][c]{\includegraphics{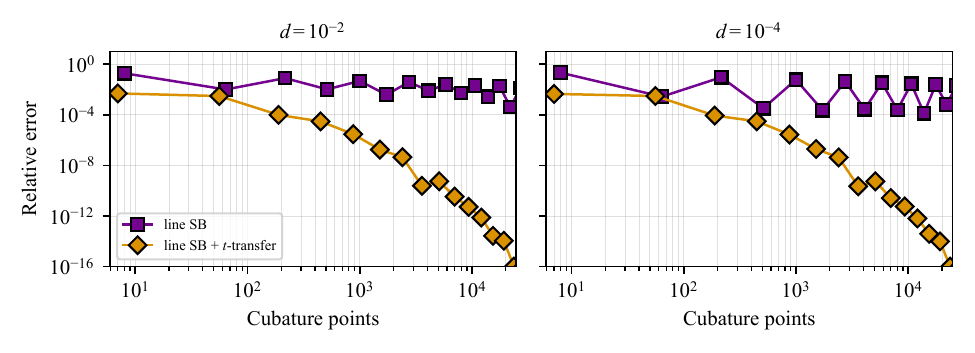}}
    \caption{}
    \label{fig:ex-line-near-boundary-convergence}
  \end{subfigure}
  \par\medskip
  \begin{subfigure}[t]{\linewidth}
    \centering
    \makebox[\linewidth][c]{\includegraphics{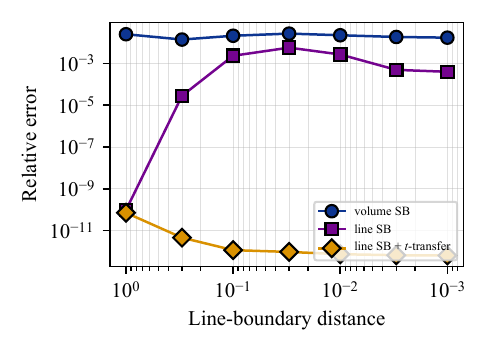}}
    \caption{}
    \label{fig:ex-line-near-boundary-distance}
  \end{subfigure}
  \caption{Near-boundary line-singularity tests for $\beta=1$.  The face-level $t$-transfer cancels the remaining
  near-singular factor on the side face $x=1$.  (a) Convergence with increasing cubature order for two line-boundary
  distances.  (b) Fixed-order error at order 12 as the line moves toward the side face $x=1$.}
  \label{fig:ex-line-near-boundary}
\end{figure}

The face-transferred rule removes the near-boundary loss of accuracy in \fref{fig:ex-line-near-boundary}.  At order 12, the
ordinary transverse-scaling rule has relative errors $2.7\times10^{-3}$ for $d=10^{-2}$ and $1.3\times10^{-4}$ for $d=10^{-4}$,
while the transferred rule yields errors below $10^{-12}$ in both cases.  In the distance sweep at the same order, the
transferred rule remains near $10^{-12}$ over the tested distances, whereas the untransferred transverse-scaling rule reaches $2.3\times10^{-3}$ at $d=10^{-1}$ and $5.6\times10^{-3}$ at $d=3\times10^{-2}$.  Thus the transverse scaling map handles
the line singularity in the swept coordinate, and the additional face-level transfer handles the near singularity that
remains in the affected boundary-patch parameter.

\FloatBarrier

\section{Conclusions}\label{sec:conclusion}

We have developed a higher-dimensional scaled boundary cubature framework for regions described by oriented boundary
representations.  The construction extends the planar SBC idea to compact regions in $\Re^d$ by sweeping each boundary
patch to a scaling center and integrating over the patch parameters together with a radial coordinate.  In three
dimensions, this yields a direct volume rule for solids bounded by affine faces, triangular patches, tensor-product
surface patches, B-spline patches, and NURBS patches.  For affine polytopes, recursively applying the same idea on each
facet produces nested tensor-product formulas over simplex sectors; in three dimensions this results in the
tetrahedral-sector formula for affine polyhedra.  Because the formulation is oriented, it also applies to nonconvex
geometries and to scaling centers that produce signed sector contributions.

The numerical examples support the main accuracy claims.  On affine polyhedra and on the four-dimensional tesseract, the
nested rule integrates polynomial data through the expected degree to near machine precision, including nonconvex and
highly faceted three-dimensional examples.  On curved solids, the B-spline and NURBS tests show rapid convergence for
both polynomial and nonpolynomial integrands, while retaining the same boundary-based construction.  These examples
demonstrate that the method is not tied to simplex or tensor-product cells in the physical domain; it only requires a
compatible oriented boundary representation and quadrature rules on the corresponding patch parameter domains.

We also introduced scaled-boundary transformations for weakly singular functions in three dimensions.  For point
singularities, choosing the singular point as the scaling center exposes the radial power in the Jacobian.  The standard
map is sufficient for the common $1/r$ and $1/r^2$ point kernels tested here, while generalized scaling and
Gauss--Jacobi quadrature recover rapid convergence for fractional powers.  The near-boundary point-source study shows
that this radial cancellation does not remove singular or nearly singular behavior that remains in the boundary-patch
parameters; a face-level $t$-transfer substantially improves those cases.  For line singularities, transverse scaling
provides the analogous cancellation in the plane normal to the line, and generalized or Gauss--Jacobi transverse rules
handle fractional transverse powers.  The near-boundary line study confirms the analogous limitation: when the line
approaches a boundary patch, a nearly singular patch factor remains even though the transverse volume singularity has
been removed.

The present work is distinct from earlier uses of SBC in two ways.  The planar SBC paper~\cite{Chin:2021:SBC} developed
the two-dimensional construction and its relation to homogeneous numerical integration, while recent work on the
VEM~\cite{Chin:2024:VEM} used SBC as a component in application-specific polygonal and polyhedral discretizations. Here,
the focus was on the integration framework itself: the arbitrary-dimensional boundary sweep, the three-dimensional
curved-boundary and affine-polyhedron specializations, the four-dimensional affine-polytope verification, and the point-
and line-singularity transformations.

These constructions also suggest a broader family of boundary-induced transformations.  The scaled-boundary rule used
here collapses all three coordinate directions to a point, while the line-singularity map preserves one coordinate and
collapses the remaining two directions to a line.  A complementary plane-based map would preserve two coordinates and
collapse only the normal direction to a reference plane.  These point-, line-, and plane-collapse maps lead to different
canonical cubature rules and different distributions of quadrature points in the physical domain.  The point-collapse
rule is natural for point singularities and oriented boundary representations, whereas line- and plane-collapse rules
may be advantageous for smooth nonpolynomial integrands or geometries with preferred directions because they can
distribute points more evenly through the domain.

Several extensions remain important.  Curved singular lines, surface singularities, trimmed NURBS patches, and implicit
boundary representations require additional geometric machinery.  Future work will also consider production
implementations of these rules for high-order embedded, contact, mortar, isogeometric, and polytopal methods, where
robust integration over curved and nonmatching regions is a recurring bottleneck.

\section*{CRediT authorship contribution statement}
Eric B. Chin: Writing -- original draft, Validation, Software, Methodology, Conceptualization.  N. Sukumar: Writing --
review \& editing, Validation, Methodology, Conceptualization.

\section*{Data availability}
Data will be made available on request.

\section*{Declaration of competing interest}
The authors declare that they have no known competing financial interests or personal relationships that could have
appeared to influence the work reported in this paper.

\section*{Declaration of generative AI and AI-assisted technologies in the manuscript preparation process}
During the preparation of this work, the authors used GPT 5.5 to edit an initial draft of the paper and expand some
sections of the manuscript based on detailed prompting from the authors.  After using this tool/service, the authors
reviewed and edited the content as needed and take full responsibility for the content of the published article.

\section*{Acknowledgements}
This work was performed under the auspices of the U.S. Department of Energy by Lawrence Livermore National Laboratory
under Contract DE-AC52-07NA27344.

\bibliography{journals,references}

@string{tog  	= {ACM Transactions on Graphics}}

@string{acme 	= {Archives of Computational Methods in Engineering}}

@string{bitnm	= {BIT Numerical Mathematics}}

@string{cmech = {Computational Mechanics}}

@string{cad  	= {Computer-Aided Design}}

@string{cagd 	= {Computer Aided Geometric Design}}

@string{cmame = {Computer Methods in Applied Mechanics and Engineering}}

@string{camwa = {Computers \& Mathematics with Applications}}

@string{eabe 	= {Engineering Analysis with Boundary Elements}}

@string{m2an 	= {ESAIM: Mathematical Modelling and Numerical Analysis}}

@string{ijnme = {International Journal for Numerical Methods in Engineering}}

@string{ijf  	= {International Journal of Fracture}}

@string{cam  	= {Journal of Computational and Applied Mathematics}}

@string{mmmas = {Mathematical Models and Methods in Applied Sciences}}

@string{pams 	= {Proceedings of the American Mathematical Society}}

@string{sinum = {SIAM Journal on Numerical Analysis}}

@string{sisc 	= {SIAM Journal on Scientific Computing}}

@article{Arioli:2019:SBP,
  author  = "Arioli, C. and Shamanskiy, A. and Klinkel, S. and Simeon, B.",
  title   = "Scaled boundary parametrizations in isogeometric analysis",
  journal = cmame,
  volume  = 349,
  pages   = "576--594",
  year    = 2019
}

@article{Antolin:2022:RNI,
  author  = {Antolin, P. and Wei, X. and Buffa, A.},
  title   = "Robust numerical integration on curved polyhedra based on folded decompositions",
  journal = cmame,
  volume  = 395,
  pages   = 114948,
  year    = 2022
}

@article{Artioli:2020:ACP,
  author  = "Artioli, E. and Sommariva, A. and Vianello, M.",
  title   = "Algebraic cubature on polygonal elements with a circular edge",
  journal = camwa,
  volume  = 79,
  number  = 7,
  pages   = "2057--2066",
  year    = 2020
}

@article{BeiraodaVeiga:2013:BPV,
  author  = "{Beir\~{a}o da Veiga}, L. and Brezzi, F. and Cangiani, A. and Manzini, G. and Marini, L. D. and Russo, A.",
  title   = "Basic principles of virtual element methods",
  journal = mmmas,
  volume  = 23,
  number  = 1,
  pages   = "199--214",
  year    = 2013
}

@article{BeiraodaVeiga:2019:TVE,
  author  = "{Beir\~{a}o da Veiga}, L. and Russo, A. and Vacca, G.",
  title   = "The virtual element method with curved edges",
  journal = m2an,
  volume  = 53,
  pages   = "375--404",
  year    = 2019
}

@article{Burman:2014:CDG,
  author  = "Burman, E. and Claus, S. and Hansbo, P. and Larson, M. G. and Massing, A.",
  title   = "{CutFEM}: {D}iscretizing geometry and partial differential equations",
  journal = ijnme,
  volume  = 104,
  number  = 7,
  pages   = "472--501",
  year    = 2015
}

@article{Cangiani:2014:HVD,
  author  = "Cangiani, A. and Georgoulis, E. H. and Houston, P.",
  title   = "$hp$-version discontinuous {G}alerkin methods on polygonal and polyhedral meshes",
  journal = mmmas,
  volume  = 24,
  number  = 10,
  pages   = "2009--2041",
  year    = 2014
}

@article{Chen:2016:ANB,
  author  = "Chen, L. and Simeon, B. and Klinkel, S.",
  title   = "A {NURBS} based {G}alerkin approach for the analysis of solids in boundary representation",
  journal = cmame,
  volume  = 305,
  pages   = "777--805",
  year    = 2016
}

@article{Chernov:2012:ECG,
  author  = "Chernov, A. and Schwab, C.",
  title   = "Exponential convergence of {G}auss-{J}acobi quadratures for singular integrals over simplices in arbitrary dimension",
  journal = sinum,
  volume  = 50,
  number  = 3,
  pages   = "1433--1455",
  year    = 2012
}

@article{Chin:2015:NIH,
  author  = "Chin, E. B. and Lasserre, J. B. and Sukumar, N.",
  title   = "Numerical integration of homogeneous functions on convex and nonconvex polygons and polyhedra",
  journal = cmech,
  volume  = 56,
  number  = 6,
  pages   = "967--981",
  year    = 2015
}

@article{Chin:2017:MCD,
  author  = "Chin, E. B. and Lasserre, J. B. and Sukumar, N.",
  title   = "Modeling crack discontinuities without element-partitioning in the extended finite element method",
  journal = ijnme,
  volume  = 110,
  number  = 11,
  pages   = "1021--1048",
  year    = 2017
}

@article{Chin:2020:AEM,
  author  = "Chin, E. B. and Sukumar, N.",
  title   = "An efficient method to integrate polynomials over polytopes and curved solids",
  journal = cagd,
  volume  = 82,
  pages   = 101914,
  year    = 2020
}

@article{Chin:2021:SBC,
  author  = "Chin, E. B. and Sukumar, N.",
  title   = "Scaled boundary cubature scheme for numerical integration over planar regions with affine and curved boundaries",
  journal = cmame,
  volume  = 380,
  pages   = 113796,
  year    = 2021
}

@article{Chin:2024:VEM,
  author  = "Chin, E. B. and Dassi, F. and Manzini, G. and Sukumar, N.",
  title   = "Numerical integration in the virtual element method with the scaled boundary cubature scheme",
  journal = ijnme,
  volume  = 125,
  number  = 20,
  pages   = "e7549",
  year    = 2024
}

@article{Duffy:1982:QOP,
  author  = "Duffy, M. G.",
  title   = "Quadrature over a pyramid or cube of integrands with a singularity at a vertex",
  journal = sinum,
  volume  = 19,
  number  = 6,
  pages   = "1260--1262",
  year    = 1982
}

@article{Duster:2020:SEM,
  author  = {D\"{u}ster, A. and Allix, O.},
  title   = "Selective enrichment of moment fitting and application to cut finite elements and cells",
  journal = cmech,
  volume  = 65,
  number  = 2,
  pages   = "429--450",
  year    = 2020
}

@article{Farah:2015:SBE,
  author  = {Farah, P. and Popp, A. and Wall, W. A.},
  title   = "Segment-based vs.\ element-based integration for mortar methods in computational contact mechanics",
  journal = cmech,
  volume  = 55,
  pages   = "209--228",
  year    = 2015
}

@article{Fries:2017:HOM,
  author  = {Fries, T. P. and Omerovi\'{c}, S. and Sch\"{o}llhammer, D. and Steidl, J.},
  title   = "Higher-order meshing of implicit geometries---{P}art {I}: {I}ntegration and interpolation in cut elements",
  journal = cmame,
  volume  = 313,
  pages   = "759--784",
  year    = 2017
}

@article{Garhuom:2022:NMF,
  author  = {Garhuom, W. and D\"{u}ster, A.},
  title   = "Non-negative moment fitting quadrature for cut finite elements and cells undergoing large deformations",
  journal = cmech,
  volume  = 70,
  number  = 5,
  pages   = "1059--1081",
  year    = 2022
}

@article{Guendelman:2003:NRB,
  author  = {Guendelman, E. and Bridson, R. and Fedkiw, R.},
  title   = "Nonconvex rigid bodies with stacking",
  journal = tog,
  volume  = 22,
  number  = 3,
  pages   = "871--878",
  year    = 2003
}

@article{Gunderman:2021:SMF,
  author  = {Gunderman, D. and Weiss, K. and Evans, J. A.},
  title   = {Spectral mesh-free quadrature for planar regions bounded by rational parametric curves},
  journal = cad,
  volume  = 130,
  pages   = 102944,
  year    = 2021
}

@article{Gunderman:2021:HAM,
  author  = {Gunderman, D. and Weiss, K. and Evans, J. A.},
  title   = "High-accuracy mesh-free quadrature for trimmed parametric surfaces and volumes",
  journal = cad,
  volume  = 141,
  pages   = 103093,
  year    = 2021
}

@article{Hesch:2011:TTD,
  author  = {Hesch, C. and Betsch, P.},
  title   = "Transient three-dimensional contact problems: mortar method.  {M}ixed methods and conserving integration",
  journal = cmech,
  volume  = 48,
  pages   = "461--475",
  year    = 2011
}

@article{Hughes:2005:IAC,
  author  = "Hughes, T. J. R. and Cottrell, J. A. and Bazilevs, Y.",
  title   = "Isogeometric analysis: {CAD}, finite elements, {NURBS}, exact geometry and mesh refinement",
  journal = cmame,
  volume  = 194,
  number  = "39--41",
  pages   = "4135--4195",
  year    = 2005
}

@article{Krishnamurthy:2011:AGA,
  author  = {Krishnamurthy, A. and McMains, S.},
  title   = "Accurate {GPU}-accelerated surface integrals for moment computation",
  journal = cad,
  volume  = 43,
  pages   = "1284--1295",
  year    = 2011
}

@article{Lasserre:1998:ICP,
  author  = {Lasserre, J. B.},
  title   = "Integration on a convex polytope",
  journal = pams,
  volume  = 126,
  number  = 8,
  pages   = "2433--2441",
  year    = 1998
}

@article{Lasserre:1999:IHF,
  author  = {Lasserre, J. B.},
  title   = "Integration and homogeneous functions",
  journal = pams,
  volume  = 127,
  number  = 3,
  pages   = "813--818",
  year    = 1999
}

@article{Loibl:2023:PQT,
  author  = {Loibl, M. and Leonetti, L. and Reali, A. and Kiendl, J.},
  title   = "Patch-wise quadrature of trimmed surfaces in isogeometric analysis",
  journal = cmame,
  volume  = 415,
  pages   = 116279,
  year    = 2023
}

@article{Lv:2019:ASD,
  author  = {Lv, J.-H. and Jiao, Y.-Y. and Feng, X.-T. and Wriggers, P. and Zhuang, X.-Y. and Rabczuk, T.},
  title   = "A series of {D}uffy-distance transformation for integrating 2{D} and 3{D} vertex singularities",
  journal = ijnme,
  volume  = 118,
  number  = 1,
  pages   = "38--60",
  year    = 2019
}

@article{Ma:2002:DTN,
  author  = {Ma, H. and Kamiya, N.},
  title   = "Distance transformation for the numerical evaluation of near singular boundary integrals with various kernels in boundary element method",
  journal = eabe,
  volume  = 26,
  pages   = "329--339",
  year    = 2002
}

@article{Marussig:2018:ICI,
  author  = {Marussig, B. and Hughes, T. J. R.},
  title   = "A review of trimming in isogeometric analysis: challenges, data exchange, and simulation aspects",
  journal = acme,
  volume  = 25,
  pages   = "1059--1127",
  year    = 2018
}

@article{Moes:1999:AFE,
  author  = {Mo\"{e}s, N. and Dolbow, J. and Belytschko, T.},
  title   = "A finite element method for crack growth without remeshing",
  journal = ijnme,
  volume  = 46,
  number  = 1,
  pages   = "131--150",
  year    = 1999
}

@article{Mousavi:2010:GDT,
  author  = {Mousavi, S. E. and Sukumar, N.},
  title   = "Generalized {D}uffy transformation for integrating vertex singularities",
  journal = cmech,
  volume  = 45,
  pages   = "127--140",
  year    = 2010
}

@article{Mousavi:2010:GGQ,
  author  = {Mousavi, S. E. and Xiao, H. and Sukumar, N.},
  title   = "Generalized {G}aussian quadrature rules on arbitrary polygons",
  journal = ijnme,
  volume  = 82,
  number  = 1,
  pages   = "99--113",
  year    = 2010
}

@article{Puso:2004:AMS,
  author  = {Puso, M. A. and Laursen, T. A.},
  title   = "A mortar segment-to-segment contact method for large deformation solid mechanics",
  journal = cmame,
  volume  = 193,
  number  = "6--8",
  pages   = "601--629",
  year    = 2004
}

@article{Saye:2015:HOQ,
  author  = {Saye, R. I.},
  title   = "High-order quadrature methods for implicitly defined surfaces and volumes in hyperrectangles",
  journal = sisc,
  volume  = 37,
  number  = 2,
  pages   = "A993--A1019",
  year    = 2015
}

@article{Schillinger:2015:TFC,
  author  = {Schillinger, D. and Ruess, M.},
  title   = "The Finite Cell Method: {A} Review in the Context of Higher-Order Structural Analysis of {CAD} and Image-Based Geometric Models",
  journal = acme,
  volume  = 22,
  pages   = "391--455",
  year    = 2015
}

@article{Scholz:2019:NIT,
  author  = {Scholz, F. and J\"{u}ttler, B.},
  title   = "Numerical integration on trimmed three-dimensional domains with implicitly defined trimming surfaces",
  journal = cmame,
  volume  = 357,
  pages   = 112577,
  year    = 2019
}

@article{Sevilla:2008:NEF,
  author  = {Sevilla, R. and Fern\'{a}ndez-M\'{e}ndez, S. and Huerta, A.},
  title   = "{NURBS}-enhanced finite element method ({NEFEM})",
  journal = ijnme,
  volume  = 76,
  pages   = "56--83",
  year    = 2008
}

@article{Sevilla:2011:TNE,
  author  = {Sevilla, R. and Fern\'{a}ndez-M\'{e}ndez, S. and Huerta, A.},
  title   = "3{D} {NURBS}-enhanced finite element method ({NEFEM})",
  journal = ijnme,
  volume  = 88,
  number  = 2,
  pages   = "103--125",
  year    = 2011
}

@article{Sommariva:2007:PGC,
  author  = {Sommariva, A. and Vianello, M.},
  title   = "Product {G}auss cubature over polygons based on {G}reen's integration formula",
  journal = bitnm,
  volume  = 47,
  pages   = "441--453",
  year    = 2007
}

@article{Sommariva:2009:GGC,
  author  = {Sommariva, A. and Vianello, M.},
  title   = "{Gauss-Green} cubature and moment computation over arbitrary geometries",
  journal = cam,
  volume  = 231,
  pages   = "886--896",
  year    = 2009
}

@article{Song:1997:TSB,
  author  = {Song, C. and Wolf, J. P.},
  title   = "The scaled boundary finite-element method---alias consistent infinitesimal finite-element cell method---for elastodynamics",
  journal = cmame,
  volume  = 147,
  number  = "3--4",
  pages   = "329--355",
  year    = 1997
}

@article{Sukumar:2015:EFE,
  author  = {Sukumar, N. and Dolbow, J. E. and Mo\"{e}s, N.},
  title   = "Extended finite element method in computational fracture mechanics: a retrospective examination",
  journal = ijf,
  volume  = 196,
  pages   = "189--206",
  year    = 2015
}

@article{Xiao:2009:ANA,
  author  = {Xiao, H. and Gimbutas, Z.},
  title   = "A numerical algorithm for the construction of efficient quadrature rules in two and higher dimensions",
  journal = camwa,
  volume  = 59,
  pages   = "663--676",
  year    = 2009
}

@article{Scholz:2021:UHO,
  author  = {Scholz, F. and J\"{u}ttler, B.},
  title   = "Using high-order transport theorems for implicitly defined moving curves to perform quadrature on planar domains",
  journal = sinum,
  volume  = 59,
  number  = 4,
  pages   = "2138--2162",
  year    = 2021
}

@article{Sommariva:2015:CMD,
  author  = {Sommariva, A. and Vianello, M.},
  title   = {Compression of multivariate discrete measures and applications},
  journal = {Numerical Functional Analysis and Optimization},
  volume  = 36,
  number  = 9,
  pages   = {1198--1223},
  year    = 2015
}

@article{Bauman:2020:CAC,
  author  = {Bauman, B. and Sommariva, A. and Vianello, M.},
  title   = {Compressed algebraic cubature over polygons with applications to optical design},
  journal = {Journal of Computational and Applied Mathematics},
  volume  = 370,
  pages   = 112658,
  year    = 2020
}

@article{Gao:2002:RIM,
  author  = {Gao, X.-W.},
  title   = {The radial integration method for evaluation of domain integrals with boundary-only discretization},
  journal = eabe,
  volume  = 26,
  number  = 10,
  pages   = {905--916},
  year    = 2002
}

@article{Gao:2011:NEA,
  author  = {Gao, X.-W. and Peng, H.-F.},
  title   = {Numerical evaluation of arbitrary singular domain integrals based on radial integration method},
  journal = eabe,
  volume  = 35,
  number  = 3,
  pages   = {587--593},
  year    = 2011
}

@article{Calabro:2018:EQR,
  author  = {Calabr\`{o}, F. and Falini, A. and Sampoli, M. L. and Sestini, A.},
  title   = {Efficient quadrature rules based on spline quasi-interpolation for application to {IGA-BEM}s},
  journal = cam,
  volume  = 338,
  pages   = {153--167},
  year    = 2018
}

@article{Aimi:2020:QFS,
  author  = {Aimi, A. and Calabr\`{o}, F. and Falini, A. and Sampoli, M. L. and Sestini, A.},
  title   = {Quadrature formulas based on spline quasi-interpolation for hypersingular integrals arising in {IgA-SGBEM}},
  journal = cmame,
  volume  = 372,
  pages   = 113441,
  year    = 2020
}

@incollection{Falini:2021:CRB,
  author    = {Falini, A. and Kandu\v{c}, T. and Sampoli, M. L. and Sestini, A.},
  title     = {Cubature rules based on bivariate spline quasi-interpolation for weakly singular integrals},
  booktitle = {Approximation Theory XVI},
  series    = {Springer Proceedings in Mathematics \& Statistics},
  pages     = {73--86},
  publisher = {Springer International Publishing},
  year      = 2021
}

@article{DegliEsposti:2023:IBH,
  author  = {Degli Esposti, B. and Falini, A. and Kandu\v{c}, T. and Sampoli, M. L. and Sestini, A.},
  title   = {{IgA-BEM} for {3D} {Helmholtz} problems using conforming and non-conforming multi-patch discretizations and {B}-spline tailored numerical integration},
  journal = camwa,
  volume  = 147,
  pages   = {164--184},
  year    = 2023
}

\end{document}